\documentclass{amsart}
\usepackage{amsmath}
\usepackage{amssymb}
\usepackage{amscd}
\usepackage{xypic}
\xyoption{all}

\begin{document}

\title[Unitary embeddings of finite loop spaces]
{Unitary embeddings of finite loop spaces}

\author{Jos\'e Cantarero}
\email{cantarero@cimat.mx}
\address{Consejo Nacional de Ciencia y Tecnolog\'ia \\
Centro de Investigaci\'on en Matem\'aticas, A.C. Unidad M\'erida \\
Parque Cient\'ifico y Tecnol\'ogico de Yucat\'an \\
Carretera Sierra Papacal-Chuburn\'a Km 5.5 \\
M\'erida, YUC 97302 \\
Mexico.}

\author{Nat\`alia Castellana}
\email{natalia@mat.uab.cat}
\address{Departament de Matem\`atiques\\
         Universitat Aut\`onoma de Barcelona\\
         Edifici Cc\\
         E--08193 Bellaterra\\
         Spain.}

\thanks{Both authors are partially supported by FEDER/MEC grant MTM2010-20692.}


\newcommand{\A}{{\mathcal A}}
\newcommand{\B}{{\mathcal B}}
\newcommand{\C}{{\mathbb C}}
\newcommand{\Ca}{{\mathcal C}}
\newcommand{\F}{{\mathbb F}}
\newcommand{\G}{{\mathcal G}}
\newcommand{\I}{{\mathbb I}}
\newcommand{\K}{{\mathbb K}}
\renewcommand{\H}{{\mathcal H}}
\newcommand{\N}{{\mathbb N}}
\newcommand{\Or}{{\mathcal O}}
\newcommand{\Q}{{\mathbb Q}}
\newcommand{\R}{{\mathbb R}}
\newcommand{\T}{{\cal T}}
\newcommand{\W}{{\mathbb W}}
\newcommand{\Z}{{\mathbb Z}}

\newcommand{\Ab}{\operatorname{Ab}\nolimits}
\newcommand{\Aut}{\operatorname{Aut}\nolimits}
\newcommand{\holim}{\operatorname{holim}}
\newcommand{\Hom}{\operatorname{Hom}\nolimits}
\newcommand{\ho}{\operatorname{ho}\nolimits}
\newcommand{\hofib}{\operatorname{hofib}\nolimits}
\newcommand{\HoTop}{\operatorname{HoTop}\nolimits}
\newcommand{\im}{\operatorname{Im}\nolimits}
\newcommand{\id}{\operatorname{id}\nolimits}
\newcommand{\ind}{\operatorname{ind}\nolimits}
\newcommand{\Inj}{\operatorname{Inj}\nolimits}
\newcommand{\Inn}{\operatorname{Inn}\nolimits}
\newcommand{\Irr}{\operatorname{Irr}\nolimits}
\newcommand{\Iso}{\operatorname{Iso}\nolimits}
\newcommand{\Ker}{\operatorname{Ker}\nolimits}
\newcommand{\Map}{\operatorname{Map}\nolimits}
\newcommand{\Mod}{\operatorname{Mod}\nolimits}
\newcommand{\Mor}{\operatorname{Mor}\nolimits}
\newcommand{\Ob}{\operatorname{Ob}\nolimits}
\newcommand{\Out}{\operatorname{Out}\nolimits}
\newcommand{\Reg}{\operatorname{Reg}\nolimits}
\newcommand{\Rep}{\operatorname{Rep}\nolimits}
\newcommand{\res}{\operatorname{res}\nolimits}
\newcommand{\rk}{\operatorname{rk}\nolimits}
\newcommand{\Spin}{\operatorname{Spin}\nolimits}
\newcommand{\Syl}{\operatorname{Syl}\nolimits}
\newcommand{\Top}{\operatorname{Top}\nolimits}
\newcommand{\Tor}{\operatorname{Tor}\nolimits}

\newcommand{\higherlim}[2]{\displaystyle\setbox1=\hbox{\rm lim}
	\setbox2=\hbox to \wd1{\leftarrowfill} \ht2=0pt \dp2=-1pt
	\setbox3=\hbox{$\scriptstyle{#1}$}
	\def\test{#1}\ifx\test\empty
	\mathop{\mathop{\vtop{\baselineskip=5pt\box1\box2}}}\nolimits^{#2}
	\else
	\ifdim\wd1<\wd3
	\mathop{\hphantom{^{#2}}\vtop{\baselineskip=5pt\box1\box2}^{#2}}_{#1}
	\else
	\mathop{\mathop{\vtop{\baselineskip=5pt\box1\box2}}_{#1}}%
	\nolimits^{#2}
	\fi\fi}
	
\newcommand{\hocolim}[2]{\displaystyle\setbox1=\hbox{\rm hocolim}
	\setbox2=\hbox to \wd1{\rightarrowfill} \ht2=0pt \dp2=-1pt
	\setbox3=\hbox{$\scriptstyle{#1}$}
	\def\test{#1}\ifx\test\empty
	\mathop{\mathop{\vtop{\baselineskip=5pt\box1\box2}}}\nolimits^{#2}
	\else
	\ifdim\wd1<\wd3
	\mathop{\hphantom{^{#2}}\vtop{\baselineskip=5pt\box1\box2}^{#2}}_{#1}
	\else
	\mathop{\mathop{\vtop{\baselineskip=5pt\box1\box2}}_{#1}}%
	\nolimits^{#2}
	\fi\fi}


\newcommand{\Ff}{{\mathcal{F}}}
\newcommand{\Ll}{{\mathcal{L}}}
\newcommand{\pcom}{^\wedge_p}

\theoremstyle{plain}
\newtheorem{theorem}{Theorem}[section]
\newtheorem*{introtheorem}{Theorem}
\newtheorem*{introproposition}{Proposition}
\newtheorem*{introcorollary}{Corollary}
\newtheorem{proposition}[theorem]{Proposition}
\newtheorem{corollary}[theorem]{Corollary}
\newtheorem{lemma}[theorem]{Lemma}

\theoremstyle{definition}
\newtheorem{definition}[theorem]{Definition}

\theoremstyle{remark}
\newtheorem{remark}[theorem]{Remark}
\newtheorem{example}[theorem]{Example}

\begin{abstract}
In this paper we construct faithful representations of
saturated fusion systems over discrete $p$-toral groups 
and use them to find conditions that guarantee the 
existence of unitary embeddings of $p$-local compact groups.
These conditions hold for the Clark-Ewing and Aguad\'e-Zabrodsky $p$-compact
groups as well as some exotic $3$-local compact groups. We also 
show the existence of unitary embeddings of finite loop spaces.
\end{abstract}

\maketitle

\section*{Introduction}

In the theory of compact Lie groups, the existence of a faithful
unitary representation for every compact Lie group is a consequence
of the Peter-Weyl theorem. This paper is concerned with the existence
of analogous representations for several objects in the literature
which are considered to be homotopical counterparts of compact Lie groups.

In 1994, W.G.~Dwyer and C.W.~Wilkerson \cite{DW} introduced $p$-compact groups. 
They are loop spaces which satisfy some finiteness properties at a particular
prime $p$. For example, if $G$ is a compact Lie group such that its
group of connected components is a finite $p$-group, then its
$p$-completion $G \pcom $ in the sense of Bousfield-Kan \cite{BK} is a $p$-compact group. 
But there are examples of $p$-compact groups which are not the $p$-completion 
of any compact Lie group, which are called exotic. Connected $p$-compact groups 
were classified in \cite{AG} and \cite{AGMV}, where a bijective correspondence 
between connected $p$-compact groups and reflection data over the $p$-adic
integers was established.

Many ideas from the theory of compact Lie groups have a homotopical
analogue for $p$-compact groups. Faithful unitary representations 
correspond to homotopy monomorphisms at the prime $p$ from the 
classifying space $BX$ of a $p$-compact group into $BU(n) \pcom$
for some $n$. A homotopy monomorphism at $p$ is map $g$ such that
the homotopy fiber $F$ of $g \pcom$ is $B\Z/p$-null, that is, 
the evaluation map $\Map(B\Z/p,F) \to F$ is a homotopy equivalence. 
For simplicity, we will call such maps $BX \to BU(n) \pcom$ 
unitary embeddings. The existence of such maps follows from the Peter-Weyl theorem, the 
classification of connected $p$-compact groups and the works \cite{C2}, \cite{C1} 
for $ p > 2 $ and \cite{Z1}, \cite{Z15}, \cite{Z2} for $ p = 2 $. 

In this article we deal with the same question for
the combinatorial structures called $p$-local compact groups,
which encode the $p$-local information of some spaces at a prime 
$p$. They were introduced in \cite{BLO3} to model $p$-completed classifying
spaces of compact Lie groups, $p$-compact groups, as well as linear
torsion groups, and they have been shown recently to model $p$-completions 
of classifying spaces of finite loop spaces \cite{BLO4} and other exotic examples \cite{G}. 

These structures generalize $p$-local finite groups \cite{BLO2}. In fact, 
they are given by a fusion system $\Ff$ over a discrete $p$-toral group $S$ 
and an associated centric linking system $\Ll$. Recall that a fusion system
$\Ff$ over $S$ is category whose objects are the subgroups of $S$ and whose morphisms
$\Hom_{\Ff}(P,Q)$ satisfy certain properties. More details can be found in 
Section \ref{Background} of this article. After introducing the notion of
homotopy monomorphism at the prime $p$ in Section \ref{Monomorphisms} and
studying some of its properties, we construct complex representations of $S$ 
which are faithful and $\Ff$-invariant in Section \ref{Representations}. A 
representation $\rho$ is fusion-preserving or $\Ff$-invariant if for any $P \leq S $ and any morphism $ f $ in 
$\Hom_{\Ff}(P,S)$, the representations $\rho_{|P}$ and $\rho_{|f(P)} \circ f$
are isomorphic. The following is Theorem \ref{FaithfulFusion}, the main theorem 
in this section. 

\begin{introtheorem}
Let $\Ff$ be a saturated fusion system over a discrete $p$-toral
group $S$. There exists a faithful unitary representation of $S$
which is $\Ff$-invariant.
\end{introtheorem}  

The importance of such representations comes from the maps $\Psi_n$
introduced in Section \ref{Embeddings} which take homotopy classes of maps $ |\Ll| \pcom \to BU(n) \pcom $ to 
its restriction to $BS$. This restriction map gives $n$-dimensional 
complex representations of $S$ which must be $\Ff$-invariant. We 
show that if a faithful fusion-preserving representation of $S$ is the image of
a map $f$ under $\Psi_n$, then $f$ is a unitary embedding. We also
say that $f$ is a unitary embedding of the $p$-local compact group $(S,\Ff,\Ll)$.

The problem is then reduced to studying the obstructions for a
(faithful) representation $\rho$ to be in the image of $\Psi _n$
coming from Wojktowiak's obstruction theory \cite{W}. These obstructions 
lie in the cohomology of the orbit category of centric radical 
subgroups with certain functors as coefficients. For our purpose it is 
enough to know whether $m\rho$ is in the image of $\Psi _{mn}$
for some $m > 0$. By stabilizing, we find that the obstructions in
even dimensions vanish and in odd dimensions we can replace our
original functors by the functor which takes a group $P$ to 
$R(P,\rho) \otimes _{\Z} \Z \pcom$, where $R(P,\rho)$ is the 
Grothendieck ring of subrepresentations of $\rho _{|P}$.

The orbit category of centric radical subgroups has finite length (see
Definition \ref{Length}) and the obstructions vanish above the length. 
Thus when the length of this category is smaller than three there are no obstructions, 
from where we obtain the following theorem, which corresponds to 
Corollary \ref{LowDepth} in the text. 

\begin{introtheorem}
Let $(S,\Ff,\Ll)$ be a $p$-local compact group such that $l(\Or(\Ff^{cr})) < 3 $.
Then there exists a unitary embedding of $(S,\Ff,\Ll)$.
\end{introtheorem}

This is the case for the Clark-Ewing and the Aguad\'e-Zabrodsky $p$-compact groups, as
we show in Section \ref{pcompact}. The existence of unitary embeddings for these $p$-compact groups 
was already shown in \cite{C2} in a different way, but these results were never published.
We would also like to add the comment here that the existence of unitary embeddings 
for $p$-local finite groups was already shown in \cite{CCM}. 

In Section \ref{Examples} we show that if $X \simeq \Omega BX$ is a finite loop space,
then $BX \pcom$ has a unitary embedding using the existence of unitary embeddings for 
$p$-compact groups and the fact that the universal cover of $BX \pcom$ is the classifying 
space of a $p$-compact group. In this section we also show the existence of unitary 
embeddings for the exotic $3$-local compact groups constructed in \cite{G} using the 
results of Section \ref{Embeddings}. 

The following theorem is a consequence of the classification of $p$-compact
groups, the Peter-Weyl theorem, the existence of unitary embeddings of generalized
Grassmannians \cite{C1} and of $DI(4)$ \cite{Z1}, \cite{Z15}, \cite{Z2}, and the results
mentioned above. 
 
\begin{introtheorem}
Let $(S,\Ff,\Ll)$ be a $p$-local compact group which models a finite loop space
or a $p$-compact group. Then there exists a unitary embedding of $(S,\Ff,\Ll)$.
\end{introtheorem}

Section \ref{pcompact} also contains a result of independent interest concerning
the relationship between the fusion systems of mapping spaces and
centralizer fusion systems. In general, if $Q$ is a fully centralized 
subgroup of $S$, it is not known whether the classifying space of the 
centralizer $p$-local compact group $(S,C_{\Ff}(Q),C_{\Ll}(Q))$ has 
the homotopy type of $\Map(BQ,|\Ll| \pcom)_{Bi}$. We have the following 
partial result, which corresponds to Proposition \ref{Centralizers}.

\begin{introproposition}
Let $X$ be a $p$-compact group, $S$ a maximal discrete $p$-toral subgroup and $\Ff$ the associated fusion
system over $S$. Let $E$ be a fully centralized subgroup of $Z(S)$. Then the fusion system $C_{\Ff}(E)$ 
is equal to the fusion system of the $p$-compact group $C_X(E)$ over $C_S(E)$.
\end{introproposition}

Finally, in Section \ref{Consequences} we consider some of the consequences of the existence of a unitary embedding
of a $p$-local compact group. We obtain finiteness results for the $p$-local cohomology
of $|\Ll| \pcom $ and a stable elements formula for the Grothendieck ring $\K(|\Ll| \pcom)$ of 
maps from $|\Ll| \pcom $ to $BU(n) \pcom $ when the length of the orbit category of centric radical subgroups 
is small. The stable elements formula comes from the map 
\[ \Psi  \colon \K (|\Ll| \pcom) \to \higherlim{\Or(\Ff ^c)}{} R(P) \]
induced by the maps $\Psi_n$ introduced in Section \ref{Embeddings}. We show that if $l(\Or(\Ff^{cr}))<3$, then
$\Psi$ is surjective and if $l(\Or(\Ff^{cr}))<2$, then $\Psi$ is an isomorphism. In particular,
this map is an isomorphism for the Clark-Ewing, the Aguad\'e-Zabrodsky $p$-compact groups and 
the exotic $3$-local compact groups of \cite{G}.
\newline

Both authors are grateful to Jesper Grodal and Jesper M\o ller for 
providing arguments which helped simplify some of the proofs.


\section{Background on $p$-local compact groups}
\label{Background}

In this section we recall the definition of a $p$-local compact group in the form given in \cite{BLO3}. 
Let $\Z /p^{\infty} \cong \Z \left[ \frac{1}{p} \right] /\Z $ denote the union of
the cyclic $p$-groups $ \Z /p^n $ under the standard inclusions.

\begin{definition}
A discrete $p$-toral group is a group $P$ with a normal subgroup $
P_0 \triangleleft P $ such that $ P_0$ is isomorphic to a finite
product of copies of $ \Z /p^{\infty} $, and $ P / P_0 $ is a finite
$p$-group. The subgroup $P_0$ will be called the identity component
of $P$.
\end{definition}

The identity component $P_0$ of a discrete $p$-toral group can be
characterized as the characteristic subgroup of all infinitely 
$p$-divisible elements in $P$. If $ P_0 \cong (\Z /p^{\infty})^k $,
we say the rank of $P$ is $k$. 

Set $ |P| = ( \rk (P), |P/P_0|)$. We regard the order of a discrete $p$-toral 
group as an element of $\N ^2 $ with the lexicographical ordering. 
That is, $|P| \leq |P'| $ if and only if $\rk (P) < \rk (P') $, 
or $ \rk (P) = \rk (P') $ and $|P/P_0| \leq |P'/P'_0| $. In 
particular, $ P' \leq P $ implies $|P| \leq |P'| $, with equality 
only if $ P' = P $.
\newline

Given two discrete $p$-toral groups $P$, $Q$, let $\Hom (P,Q)$
denote the set of group homomorphisms from $P$ to $Q$, and let $\Inj
(P,Q)$ denote the set of monomorphisms. If $P$ and $Q$ are subgroups
of a larger group $S$, then $ \Hom _S(P,Q)\subseteq \Inj (P,Q)$
denotes the subset of homomorphisms induced by conjugation by
elements of $S$, and $\Aut _S(P)$ the group of automorphisms 
of $P$ induced by conjugation in $S$

\begin{definition}
\label{SaturatedFusion}

A fusion system $\Ff$ over a discrete $p$-toral group $S$ is a
subcategory of the category of groups whose objects are the subgroups 
of $S$, and whose morphism sets $ \Hom _{\Ff}(P,Q)$ satisfy the following conditions:

\begin{itemize}
\item[(a)] $ \Hom_S(P,Q) \subseteq \Hom_{\Ff}(P,Q) \subseteq \Inj(P,Q)$ for all $ P,Q \leq S $.
\item[(b)] Every morphism in $\Ff$ factors as an isomorphism in $\Ff$ followed by an inclusion.
\end{itemize}

\end{definition}

Let $\Iso _{\Ff}(P,Q)$ be the subset of $\Hom_{\Ff}(P,Q)$ of 
isomorphisms. We also use the notation $ \Aut _{\Ff}(P) = \Iso _{\Ff}(P,P) $ 
and $ \Out_{\Ff}(P) = \Aut _{\Ff}(P)/ \Inn (P)$. Two subgroups $P,P' \leq S $
are called $\Ff$-conjugate if $\Iso _{\Ff}(P,P') \neq \emptyset $.

The fusion systems we consider in this article will all satisfy the following saturation
condition. Here, and throughout the paper, we write $\Syl_p(G)$ for the set of Sylow 
$p$-subgroups of $G$. Also, for any $ P \leq G $ and any $ g \in N_G(P) $, $c_g \in \Aut (P)$ denotes the
automorphism $ c_g(x) = gxg^{-1} $.

\begin{definition}

Let $\Ff$ be a fusion system over a discrete $p$-toral group $S$.

\begin{itemize}
\item A subgroup $ P \leq S $ is fully centralized in $\Ff$ if $|C_S(P)| \geq |C_S(P')| $ for all $ P' \leq S $
that are $\Ff$-conjugate to $P$.
\item A subgroup $ P \leq S $ is fully normalized in $\Ff$ if $|N_S(P)| \geq |N_S(P')| $ for all $ P' \leq S $
that are $\Ff$-conjugate to $P$.
\item $\Ff$ is a saturated fusion system if the following three conditions hold:

\begin{itemize}
\item[(I)] For each $ P \leq S $ which is fully normalized in $\Ff$, $P$ is fully centralized in $\Ff$,
$\Out_{\Ff}(P)$ is finite and $\Out_S(P) \in \Syl_p(\Out_
{\Ff}(P))$.
\item[(II)] If $ P \leq S $ and $\phi \in \Hom _{\Ff}(P,S)$ are such that $\phi P $ is fully centralized, and if we set
\[ N_{\phi} = \{ g \in N_S(P) \mid \phi c_g \phi ^{-1} \in \Aut _S(\phi P) \} \]
then there is $\bar{\phi} \in \Hom _{\Ff}(N_{\phi},S)$ such that $
\bar{\phi} |_P = \phi $.
\item[(III)] If $ P_1 \leq P_2 \leq P_3 \leq \ldots $ is an increasing
sequence of subgroups of $S$, with $ P_{\infty} = \mathop{\cup}
\limits_{n=1}^{\infty} P_n $, and if $\phi \in \Hom (P_{\infty},S) $
is any homomorphism such that $ \phi _{|P_n} \in \Hom _{\Ff}(P_n,S)
$ for all $n$, then $ \phi \in \Hom _{\Ff}(P_{\infty},S)$.
\end{itemize}

\end{itemize}

\end{definition}

The motivating example for this definition is the fusion system of a connected compact Lie group $G$. 
If $T$ is a maximal torus of $G$, let $W_p$ be the $p$-Sylow subgroup of the Weyl group $W_G(T)
= N_G(T)/T$. This subgroup determines an extension $N_p$ of $W_p$ by $T$. The proof of Proposition 
9.3 (b) in \cite{BLO3} shows that any compact Lie group $G$ has a maximal discrete 
$p$-toral subgroup $S$ which can found as a discrete subgroup of $N_p$ and that it is unique up to $G$-conjugacy. 
We let $\Ff _S(G)$ be the fusion system over $S$ defined by setting $\Hom _{\Ff_S(G)}(P,Q) = \Hom _G(P,Q)$ for all $ P,Q \leq S $.

\begin{lemma}[Lemma 2.4 in \cite{BLO3}]
\label{WeylGroup}

Let $\Ff$ be a saturated fusion system over a discrete $p$-toral
group $S$ with connected component $ T = S_0 $. Then the following
hold for all $ P \leq T $.

\begin{itemize}
\item For every $ P' \leq S $ which is $\Ff$-conjugate to $P$ and
fully centralized in $\Ff$, $ P' \leq T $ and there exists some $ w
\in \Aut _{\Ff} (T) $ such that $ w _{|P} \in \Iso _{\Ff} (P,P') $.
\item Every $ \phi \in \Hom _{\Ff} (P,T) $ is the restriction of
some $ w \in \Aut _{\Ff}(T) $.
\end{itemize}

\end{lemma}

\begin{definition}
Let $\Ff$ be any fusion system over a discrete $p$-toral group $S$.
A subgroup $ P \leq S $ is called $\Ff$-centric if $P$ and all of its
$\Ff$-conjugates contain their $S$-centralizers. A subgroup $ Q \leq S $
is called $\Ff$-radical if $\Out_{\Ff}(Q)$ contains no nontrivial normal $p$-subgroup. 
\end{definition}

Note that $\Out_{\Ff}(Q)$ is finite for all $Q \leq S$ by Proposition 2.3
in \cite{BLO3}, so this definition makes sense.

We will denote by $\Ff ^c $ the full subcategory of $\Ff $ whose objects are the $\Ff$-centric subgroups of $S$
and by $\Ff^{cr}$ the full subcategory whose objects are the $\Ff$-centric radical subgroups of $S$.

\begin{theorem}[Alperin's fusion theorem]
\label{Alperinfusionthm}

Let $\Ff$ be a saturated fusion system over a discrete $p$-toral
group $S$. Then for each $ \phi \in \Iso _{\Ff} (P,P')$, there exist
sequences of subgroups of $S$
\[ P = P_0, P_1, \ldots , P_k = P' \text{ and  } Q_1, Q_2, \ldots , Q_k \]
and elements $ \phi _i \in \Aut _{\Ff} (Q_i) $, such that

\begin{itemize}
\item[(a)] $Q_i $ is fully normalized in $\Ff$, $\Ff$-radical, and
$\Ff$-centric for each $i$ ;
\item[(b)] $P_{i-1}$, $P_i \leq Q_i $ and $ \phi _i(P_{i-1}) = P_i $ for
each $i$ ; and
\item[(c)] $ \phi = \phi _k |_{P_{k-1}} \circ \phi _{k-1}|_{P_{k-2}} \circ \ldots \circ \phi _1|_{P_0}$.
\end{itemize}

\end{theorem}

\begin{proof}
See \cite{BLO3}, Theorem 3.6.
\end{proof}

\begin{definition}
\label{Linking}
Let $\Ff $ be a fusion system over the $p$-group $S$. A centric
linking system associated to $\Ff$ is a category $\Ll$ whose objects
are the $\Ff $-centric subgroups of $S$, together with a functor
\[ \pi  \colon \Ll \rightarrow \Ff ^c \]
and monomorphisms $ P \stackrel{\delta_P}{\longrightarrow} \Aut _{\Ll}(P)$ 
for each $\Ff$-centric subgroup $ P \leq S $, which satisfy the following conditions:

\begin{itemize}
\item[(A)] $\pi$ is the identity on objects and surjective on morphisms. More precisely, for each pair of objects $ P,Q \in \Ll $,
$Z(P)$ acts freely on $\Mor _{\Ll}(P,Q)$ by composition (upon
identifying $Z(P)$ with $\delta _P(Z(P)) \leq \Aut _{\Ll}(P)$), and
$\pi$ induces a bijection
\[ \Mor _{\Ll}(P,Q)/Z(P) \stackrel{\cong}{\longrightarrow} \Hom _{\Ff}(P,Q) \]
\item[(B)] For each $\Ff$-centric subgroup $ P \leq S $ and each $ g \in P $, $\pi $ sends the 
element $\delta _P(g) \in \Aut_{\Ll}(P)$ to $c_g \in \Aut _{\Ff}(P) $.
\item[(C)] For each $ f \in \Mor _{\Ll}(P,Q) $ and each $g \in P$, the following square commutes in $\Ll$:
\[
\diagram P \rto^{f} \dto_{\delta _P(g)} & Q \dto^{\delta _Q(\pi
(f)(g))} \cr P \rto^f & Q
\enddiagram
\]
\end{itemize}

\end{definition}

\begin{definition}

A $p$-local compact group is a triple $(S,\Ff,\Ll)$, where $\Ff$ is
a saturated fusion system over the discrete $p$-toral group $S$ and
$\Ll$ is a centric linking system associated to $\Ff$. The
classifying space of the $p$-local finite group $(S,\Ff,\Ll)$ is the
space $|\Ll|^{\wedge}_p $. We will refer to $S$ as the Sylow of 
$|\Ll|^{\wedge}_p $.

\end{definition}

The orbit category $\Or (\Ff)$ is the category whose objects are the
subgroups of $S$ and whose morphisms are given by
\[ \Hom _{\Or (\Ff)}(P,Q) = \Rep _{\Ff} (P,Q) = \Hom _{\Ff}(P,Q) / \Inn (Q) . \] 
For any full subcategory $\Ff _0 $ of $\Ff$ we also consider $\Or (\Ff _0) $, the
full subcategory of $\Or (\Ff) $ whose objects are those of $\Ff
_0$. Let $B  \colon \Or (\Ff_0) \to \HoTop$ be the classifying space functor
defined on the homotopy category of topological spaces.

\begin{proposition}[Proposition 4.6 in \cite{BLO3}]
\label{Rigidification}

Fix a saturated fusion system $\Ff $ over a discrete $p$-toral group
$S$, and let $\Ff _0 \subseteq \Ff ^c $ be any full subcategory. For
any linking system $\Ll _0 $ associated to $\Ff _0 $, the left
homotopy Kan extension $ \widetilde{B} \colon \Or (\Ff _0) \to \Top $ of the
constant functor $ \Ll _ 0 \to \Top $
along the projection $ \widetilde{\pi}_0 \colon \Ll _0 \to \Or (\Ff _0) $ is
a rigidification of $B$, and there is a homotopy equivalence:
\[   |\Ll _0| \simeq \hocolim{\Or (\Ff _0)}{} \widetilde{B} \]
\end{proposition}

It is possible to construct $p$-local compact groups associated to compact Lie groups, 
$p$-compact groups and some locally finite discrete groups in such a way that their 
classifying spaces are homotopy equivalent to the $p$-completion of the classifying 
spaces of the given groups \cite{BLO3}.

\section{Homotopy monomorphisms}
\label{Monomorphisms}

In \cite{DW}, W.G.~Dwyer and C.W.~Wilkerson introduced the notion of a monomorphism between $p$-compact groups. 
A $p$-compact group is a triple $(X,BX,e)$, where $X$ is an $\F_p$-finite space (that is, $H^*(X;\F _p)$ is finite), 
$BX$ is a $p$-complete connected pointed space and $ e \colon X \to \Omega BX$ is a homotopy equivalence. A 
pointed map $f \colon BX \to BY$ is a monomorphism if the homotopy fiber $F$ of $f$ is $\F_p$-finite, that is, 
$H^*(F; \F_p)$ is finite. 

But in general, if $G$ is a finite group, the constant map $ * \to BG \pcom $ may not be a monomorphism 
in this sense. For instance, Proposition 5.1.1 of \cite{L} implies that if $H$ is a finite $2$-perfect
group whose $2$-Sylow subgroup is a dihedral group, then the homotopy fiber $\Omega ( BH _2^{\wedge} ) $ is 
not $\F_2$-finite. A particular example is given by the group $A_7$ with $2$-Sylow $D_8$. 

However, if $G$ is a finite group and $ c : B\Z/p \to BG \pcom$ is the trivial map, then
\[ \Map(B\Z/p,BG \pcom)_c \simeq BC_G(\{ 1 \}) \pcom = BG \pcom \]
where the homotopy equivalence follows from Proposition 7.5 of \cite{BrK}. By Remark \ref{LoopLocal} below,
the space $\Map_*(B\Z/p,\Omega (BG \pcom))$ is contractible. That is, $\Omega (BG \pcom)$ is $B\Z/p$-null in 
the sense of \cite{F} (see Definition \ref{Null} and Remark \ref{LoopLocal}).

\begin{definition}
\label{Null}
Let $A$ be a pointed space. We say that a pointed space $Z$ is $A$-null if the natural map
$ Z \to \Map(A,Z)$ is a homotopy equivalence. We say that a map $f \colon X \to Y$ 
is a homotopy monomorphism at $p$ if the homotopy fiber $F$ of $f \pcom$ is $B\Z/p$-null.
\end{definition}

The space $Z$ is a $A$-null if and only if the evaluation map $\Map(A,Z) \to Z$ is a homotopy 
equivalence. Note then that if $Z$ is a connected pointed space, then $Z$ is $A$-null if and 
only if $\Map_*(A,Z)$ is contractible. 

The following proposition shows that Definition \ref{Null} is consistent with the previous
definition of monomorphism for connected $p$-compact groups in \cite{DW}. Recall
that an unstable module $M$ over the Steenrod algebra $\A$ is called locally finite if for 
any element $x \in M$, the $\A$-module $M$ generated by $x$ is finite.

\begin{proposition}
\label{PCompactMono}
let $f \colon BX \to BY$ be a homomorphism of connected $p$-compact groups with homotopy fiber $Y/X$. 
Then $Y/X$ is $\F_p$-finite if and only if it is $B\Z/p$-null.  
\end{proposition}

\begin{proof}
Consider the fiber sequence
\[ X \to Y \to Y/X \to BX \to BY \]
The spaces $X$ and $Y$ are nilpotent because they are loop spaces. By Proposition III.5.5 of \cite{BK},
$Y/X$ is also nilpotent. This sequence also implies that $H^*(Y/X;\F_p)$ is a finitely generated
$H^*(BX;\F_p)$-module. Since $H^*(BX;\F_p)$ is finitely generated as an $\F_p$-algebra, $H^*(Y/X;\F_p)$
is finitely generated as an $\F_p$-algebra. In particular, $Y/X$ has finite type over $\F_p$, that is, each
$H^j(Y/X;\F_p)$ is finite. Note that $BX$ and $BY$ are $p$-complete and $BX$ is simply connected, 
so $Y/X$ is $p$-complete by Lemma II.4.8 of \cite{BK}. As $Y$ is connected, so is $Y/X$. 

H.~Miller's proof of the Sullivan conjecture shows that if $F$ is connected, nilpotent and $\F_p$-finite, 
then it is also $B\Z/p$-null (Theorem C of \cite{Mi}), so if $Y/X$ is $\F_p$-finite, then it 
is $B\Z/p$-null. 

If $Y/X$ is $B\Z/p$-null, then $\Map(B\Z/p,Y/X) \simeq Y/X$. In particular, $\Map(B\Z/p,Y/X)$ is 
$p$-complete. Since $Y/X$ has finite type over $\F_p$, we can use Corollary 3.4.3 of \cite{La}
and obtain
\[ T_{\Z/p}(H^*(Y/X;\F_p)) \cong H^*(\Map(B\Z/p,Y/X);\F_p) \cong H^*(Y/X;\F_p) \]
where $T_{\Z/p}$ is Lannes' T-functor introduced in \cite{La}. By Theorem 6.2.1 in 
\cite{S}, it follows that $H^*(Y/X;\F_p)$ is locally finite and in particular, all 
of its elements are nilpotent. Because $H^*(Y/X;\F_p)$ is finitely
generated as an $\F_p$-algebra, it must be finite. 
\end{proof}

\begin{remark}
\label{LoopLocal}
Let $A$ and $X$ be connected spaces and let $c \colon A \to X$ be the constant map. Then $\Map(A,X)_c \simeq X$ via the 
evaluation map if and only if $\Map_*(A,\Omega X)$ is contractible. Note that a space is $A$-null if and only if each
of its connected components is $A$-null and all the connected components of $\Omega X$ are homotopy equivalent. Therefore
$\Map(A,X)_c \simeq X$ via the evaluation map if and only if $\Omega X$ is $A$-null. In particular, by Theorem 6.3(c) 
of \cite{BLO3}, if $|\Ll| \pcom $ is the classifying space of a $p$-local finite group, then $\Omega (|\Ll| \pcom) $ is 
$B\Z/p$-null. Hence the constant map $* \to |\Ll| \pcom$ is a homotopy monomorphism at $p$.
\end{remark}

From now on, we fix a prime $p$ and use homotopy monomorphism instead of homotopy monomorphism at $p$. The
next lemma contains some auxiliary results that will be used in Theorem \ref{HomotopyMonomorphism}.

\begin{lemma}
\label{PropertiesMono}
Let $A$ be a connected pointed space. 
\begin{itemize}
\item[(a)] If $Y$ is homotopically discrete, then $Y$ is $A$-null.
\item[(b)] A homotopy equivalence is a homotopy monomorphism.
\item[(c)] Let $F \to E \to B$ be a homotopy fiber sequence. If $B$ and $E$ are $A$-null, then $F$ is $A$-null. If
$B$ is $A$-null and connected and $F$ is $A$-null, then $E$ is $A$-null.
\item[(d)] Let $ g \colon Y \to Z$ and $f \colon W \to Y$. If $g$ and $g \circ f$ are homotopy monomorphisms, then
so is $f$. If $f$ and $g$ are homotopy monomorphisms and the homotopy fiber of $g \pcom$ is connected, then 
$g \circ f$ is a homotopy monomorphism.
\item[(e)] The space $X$ is $A$-null if and only if $\Omega X$ is $A$-null and every map $A \to X$ is nullhomotopic.
\item[(f)] Let $H$ and $G$ be discrete $p$-toral groups and let $ \alpha \colon H \to G$ be a group homomorphism. Then $B\alpha$ is a homotopy monomorphism if and only if
$\alpha$ is injective.
\item[(g)] If $X$ is $p$-good, the $p$-completion map $X \to X \pcom $ is a homotopy monomorphism.
\end{itemize}
\end{lemma}

\begin{proof}
Parts (a), (b) and (g) follow directly from the definitions. Part (c)
follows from Section A.8, e.6 of \cite{F} together with the observations
that $F$ is the homotopy fiber over the basepoint of $B$ and that a space
is $A$-null if and only if each of its connected components is $A$-null. 

Part (d) now follows from (c) and the homotopy fiber sequence of homotopy fibers
\[ \hofib (f \pcom) \to \hofib (g \pcom \circ f \pcom) \to \hofib (g \pcom). \]

If $X$ is $A$-null, then $\Omega X$ is $A$-null by part (c) and
every map $A \to X$ is nullhomotopic because the composition of
the evaluation map $\Map(A,X) \to X$ and the constant map $ X \to \Map(A,X)$
is homotopic to the identity. 

On the other hand, $\Omega X$ is $A$-null if and only if $\Omega X_j$ is $A$-null
for each connected component $X_j$ of $X$. Then $\Map(A,X_j)_c \simeq X_j$ via the 
evaluation map by Remark \ref{LoopLocal}. And since every map $A \to X_j$ is 
nullhomotopic, the space $\Map(A,X_j)$ is connected and so equal to $\Map(A,X_j)_c$. 
Therefore $\Map(A,X_j) \simeq X_j$ via the evaluation map and $\Map(A,X) \simeq X$
via the evaluation map, that is, $X$ is $A$-null. This completes the proof of part (e)

Let $F$ be the homotopy fiber of $B\alpha \pcom$. By Remark \ref{LoopLocal}, the
spaces $\Omega (BH \pcom)$ and $\Omega (BG \pcom)$ are $B\Z/p$-null. Part (c) applied
to the homotopy fiber sequence $\Omega F \to \Omega (BH \pcom) \to \Omega (BG \pcom) $
yields that $\Omega F$ is $B\Z/p$-null. Note that the proof of Lemma 1.10 in \cite{BLO3}
shows that $[B\Z/p,BS] = [B\Z/p,BS \pcom]$ if $S$ is discrete $p$-toral.

If $\alpha \colon H \to G$ is not injective, then there exists $B\Z/p \to BH$ not nullhomotopic
such that the composition $ B\Z/p \to BH \to BG$ is nullhomotopic. Therefore $B\Z/p \to BG \pcom$
is nullhomotopic and by the comment in the paragraph above, $B\Z/p \to BH \pcom $ is not nullhomotopic.
This determines a map $B\Z/p \to F$ which is not nullhomotopic, and so $F$ is not $B\Z/p$-null.

On the other hand, if $F$ is not $B\Z/p$-null, then there exists $\Sigma^i B\Z/p \to F$ not nullhomotopic
for some $i \geq 0$. Since $\Omega F$ is $B\Z/p$-null, this can only happen if $i=0$. The composition
$B\Z/p \to F \to BH \pcom$ can not be nullhomotopic because otherwise it would lift to a map $B\Z/p \to \Omega (BG \pcom)$,
which must be nullhomotopic since $\Omega (BG \pcom)$ is $B\Z/p$-null. But the composition $ B\Z/p \to BG \pcom$ is nullhomotopic. 
By the same argument as before, these maps are induced by $B\Z/p \to BG$ not nullhomotopic 
and $ B\Z/p \to BG \to BH$ nullhomotopic. Therefore there is a monomorphism $\Z/p \to G$ such that 
$\Z/p \to G \to H$ is trivial, that is, $\alpha$ is not injective. This proves part (f).
\end{proof}

Let $f \colon |\Ll_1| \pcom \to |\Ll_2| \pcom$ be a map between classifying spaces of $p$-local compact groups. By Theorem 6.3(a)
of \cite{BLO3}, there is a homomorphism between their Sylow subgroups $\rho \colon S_1 \to S_2$ such that $f \circ \Theta_1 \simeq 
\Theta_2 \circ B\rho$, where $\Theta_i \colon BS_i \to | \Ll_i | \pcom $ for $ i = 1,2$ are the canonical inclusions induced 
by the structure morphisms $\delta_{S_i} \colon S_i \to \Aut_{\Ll_i}(S_i)$ from Definition \ref{Linking}. The main result of 
this section is the following.

\begin{theorem}
\label{HomotopyMonomorphism}
Let $f \colon |\Ll_1| \pcom \to |\Ll_2| \pcom$ be a map between classifying spaces of $p$-local compact groups which is surjective
on the fundamental groups and let $\rho \colon S_1 \to S_2 $ be a homomorphism between their Sylow subgroups such that 
$f \circ \Theta_1 \simeq \Theta_2 \circ B\rho$. Then $f$ is a homotopy monomorphism if and only if $\rho$ is injective.
\end{theorem}

\begin{proof}
We first show that the map induced by canonical inclusion $\Theta \colon BS \to |\Ll| \pcom$ of a Sylow subgroup 
into the classifying space of a $p$-local compact group is a homotopy monomorphism. Let $F$ the the homotopy 
fiber of $\Theta \pcom$. There is a homotopy fiber sequence 
\[\Omega F \to \Omega (BS \pcom) \to \Omega (|\Ll| \pcom) . \] 
Recall that $\Omega (|\Ll| \pcom)$ and $\Omega (BS \pcom) $ are $B\Z/p$-null by Remark \ref{LoopLocal}. Then $\Omega F$ 
is also $B\Z/p$-null by part (c) of Lemma \ref{PropertiesMono}. In order to prove that $F$ is $B\Z/p$-null, we only 
need to show that any map $ B\Z/p \to F $ is nullhomotopic.

Let $\alpha \colon B\Z/p \to F$ and let us denote by $i \colon F \to BS \pcom$ the map from the homotopy fiber to $BS \pcom$. 
Since $\Theta \circ i \circ \alpha \colon B\Z/p \to | \Ll | \pcom$ is nullhomotopic, the map $i \circ \alpha \colon B\Z/p \to BS \pcom$ 
is nullhomotopic (\cite[Theorem 6.3(a)]{BLO3} and $[B\Z/p,BS] = [B\Z/p,BS \pcom]$). Therefore $\alpha$ lifts to a map 
$ \widetilde{\alpha} \colon B\Z/p \to \Omega (|\Ll| \pcom)$ which is nullhomotopic because $\Omega (|\Ll| \pcom)$ is $B\Z/p$-null. 
And so $\alpha$ is nullhomotopic.

Now let us consider the general case. Assume that $f \colon |\Ll_1| \pcom \to |\Ll_2| \pcom$ is a homotopy monomorphism which
is surjective on the fundamental groups and let $ i \colon F \to | \Ll_1 | \pcom $ be the inclusion of the homotopy fiber of $f = f \pcom$. 
From the long exact sequence of homotopy groups, $F$ must be connected. By Lemma \ref{PropertiesMono} (d), and using that $\Theta_1$ is a homotopy 
monomorphism and $F$ is connected, the composition $f \circ \Theta_1$ is also a homotopy monomorphism. 
Since $f \circ \Theta_1 \simeq \Theta_2 \circ B\rho$ we obtain that $B\rho$ is a homotopy monomorphism using Lemma \ref{PropertiesMono} (d). 
Therefore $\rho$ is injective. 

Assume $\rho$ is injective and as before, let $ i \colon F \to | \Ll_1 | \pcom $ be the inclusion of the homotopy fiber of $f$. 
The space $\Omega F$ is $B\Z/p$-null because both $ \Omega (|\Ll_1 | \pcom)$ and $\Omega (| \Ll_2 | \pcom) $ 
are $B\Z/p$-null. It remains to show that any map $\alpha \colon B\Z/p \to F$ is nullhomotopic. By Theorem 6.3(a) in \cite{BLO3} there is a group 
homomorphism $\rho_1 \colon \Z/p \to S_1$ such that $i \circ \alpha \simeq \Theta_1 \circ B\rho_1$. Since $f \circ i \circ \alpha$ 
is nullhomotopic, so is $\Theta_2 \circ B(\rho\circ \rho_1)$ and by \cite[Theorem 6.3(a)]{BLO3}, $\rho \circ \rho_1 $ is 
the trivial morphism. Since $\rho$ is injective, $\rho_1$ is trivial and $i \circ \alpha$ is nullhomotopic. 
Then $\alpha$ lifts to $\widetilde{\alpha} \colon B\Z/p \to \Omega (|\Ll_2| \pcom) $ which is nullhomotopic 
since $\Omega (|\Ll_2| \pcom)$ is $B\Z/p$-null, and so $\alpha$ is nullhomotopic.
\end{proof}

\section{Fusion-preserving representations}
\label{Representations}

The main result in this section is Theorem \ref{FaithfulFusion}, which
shows the existence of a unitary representation of $S$ which is fusion-preserving
in the sense described below. All representations in this section are complex representations.

\begin{definition}
Let $\Ff$ be a saturated fusion system over a discrete $p$-toral
group $S$. A representation $ \rho \colon S \rightarrow GL(V) $ is
fusion-preserving or $\Ff$-invariant if for any $ P \leq S $ and any $ f \in \Hom
_{\Ff}(P,S) $, we have $ [\rho _{|P}] = [\rho_{|f(P)} \circ f] $ in $\Rep(P,GL(V))$.
\end{definition}

\begin{remark}
\label{AlperinCharacters}
Note that $\rho$ is fusion-preserving if and only if $[\rho _{|Q}] = [\rho_{|Q} \circ \phi]$
in $\Rep(Q,GL(V))$ for any $\Ff$-centric radical $Q$ and any $\phi \in \Aut_{\Ff}(Q)$. 
For any morphism $ \psi \colon P \to S $ in $\Ff$ is the composition of inclusions and
automorphisms of $\Ff$-centric radical subgroups of $S$ by Alperin's fusion theorem (Theorem \ref{Alperinfusionthm}). 
Hence if we denote by $\Rep ^{\Ff}(S,U(n))$ the set of $\Ff$-invariant $n$-dimensional representations
of $S$, we have
\[ \Rep ^{\Ff}(S,U(n)) \cong \higherlim{\Or(\Ff ^c)}{} \Rep (P,U(n)) \cong \higherlim{\Or(\Ff ^{cr})}{} \Rep (P,U(n)) \]
\end{remark}

\begin{definition}

We say that a class function $ \chi \colon S \rightarrow \C $ is
fusion-preserving or $\Ff$-invariant if for any $ P \leq S $ and any $ f \in \Hom
_{\Ff}(P,S) $, we have $ \chi_{|P} = \chi _{ |f(P) } \circ f $.

\end{definition}

\begin{lemma}
\label{FusionCharacters}

Let $ V$ be a unitary representation of $S$. Then $V$ is fusion-preserving 
if and only if its character $ \chi _V $ is fusion-preserving.

\end{lemma}

\begin{proof}
This follows from Corollary B.6 of \cite{Z15}, which shows that
two unitary representations of a locally finite group are isomorphic
as representations if and only if their characters are equal.
\end{proof}

Let us consider a discrete $p$-toral group $S$ of rank $n$ with $ T = S_0 $. Given a representation $\rho \colon T \to GL(V)$, let $\ind_T^S(V)$ be the finite-dimensional vector space 
$\C S \otimes _{\C T} V $, where $ s \otimes v = st^{-1} \otimes \rho(t)(v) $ for $ t \in T $. Note that this is finite dimensional since $T$ has finite index in $S$. 
Now $S$ acts on $\ind_T^S(V)$ by $ s' \cdot s \otimes v  =  s's \otimes v $. This action is well defined and linear, so it determines a finite-dimensional representation $ \ind_T^S(\rho) \colon S \to GL(\C S \otimes _{\C T} V)$.

\begin{proposition}
\label{InductionTorus}
Let $\Ff$ be a saturated fusion system over $S$ and $\rho \colon T \to U(m)$ a representation that is invariant 
under $\Aut_{\Ff}(T)$ and such that $\chi_{\rho}(x)=0$ if $ x \in P \cap T - P_0 $ for any $P \in \Ff^{cr}$.
Then $\ind_T^S(\rho) \colon S \to U(m|S/T|)$ is a fusion-preserving representation of $S$. Moreover, if $\rho$ 
is injective, so is $ \ind_T^S(\rho)$. 
\end{proposition}

\begin{proof}
Let $ \widetilde{\rho} = \ind_T^S(\rho) $. We will show that $ \widetilde{\rho}$ is $\Ff$-invariant. 
By Remark \ref{AlperinCharacters}, it suffices to show that $ \widetilde{\rho}_{|Q}$ and 
$ \widetilde{\rho}_{|Q} \circ \phi$ are isomorphic representations for any $\Ff$-centric radical $Q$ 
and any $\phi \in \Aut_{\Ff}(Q)$. 

Let $X$ be a set of coset representatives for the elements of $S/T$ and $ h \in Q$. If $h \in Q_0$, then
\[ \chi _{\widetilde{\rho}}(h) = \sum _{s \in X}\chi_{\rho}(s^{-1}hs) = |S/T| \chi_{\rho}(h) \]
where the second equality holds because $ c_s $ belongs to $\Aut_S(T) \subseteq \Aut _{\Ff} (T) $ for all $ s \in S $ 
and we are assuming that $\rho$ is $\Aut_{\Ff}(T)$-invariant. On the other hand, if $ h \in Q \cap T - Q_0$ we have
similarly
\[ \chi _{\widetilde{\rho}}(h) = \sum _{s \in X}\chi_{\rho}(s^{-1}hs) = |S/T| \chi_{\rho}(h) = 0 \]
And if $ h \in Q - Q \cap T $, we have $\chi _{\widetilde{\rho}}(h)=0$ since $\widetilde{\rho}$ is a 
representation induced from a representation of $T$. So we showed 
\[ \chi _{\widetilde{\rho}_{|Q}} (h) = \left\{\begin{array}{ll}
                                       |S/T|\chi_{\rho}(h) & \text{ if $h \in Q_0$} \\
                                        0 & \text{ otherwise, } \end{array} \right. \]

Note that $Q_0$ is the characteristic subgroup of infinitely $p$-divisible elements of $Q$. Thus, $ h \in Q_0 $ if and only 
if $ \phi (h) \in Q_0 $. In fact, $ \phi $ restricts to a homomorphism $ \widetilde{\phi} \colon Q_0 \rightarrow Q_0 $ in $\Ff$. 
By Lemma \ref{WeylGroup}, $\widetilde{\phi}$ extends to a map $ \psi \in \Aut_{\Ff}(T) $. Therefore, for $ h \in Q_0 $:
\[ \chi_{\rho}(\phi(h))= \chi_{\rho}(\widetilde{\phi}(h)) = \chi_{\rho}(\psi(h)) = \chi_{\rho}(h) \]
And if $h \in Q - Q_0$, then $\phi(h) \in Q - Q_0$ and so $\chi_{\widetilde{\rho}}(\phi(h)) = 0 $. Thus 
$ \chi_{\widetilde{\rho}_{|Q}} = \chi_{\widetilde{\rho}_{|Q} \circ \phi} $ and so $ \widetilde{\rho}_{|Q}$ and 
$ \widetilde{\rho}_{|Q} \circ \phi$ are isomorphic representations by Corollary B.6 of \cite{Z15}. The last statement is
a general property of induced representations.
\end{proof}

In order to use this proposition, we need to find representations of $T$ whose
characters satisfy the vanishing condition in the hypothesis. Given $P \leq S$,
we can construct a unitary representation of $T$ whose character vanishes on
$ P \cap T - P_0$ using the following lemma. 

\begin{lemma}
\label{EachSubgroup}
Let $ K \leq T$. There is a unitary representation $ \rho _K$ of $T$ such that $\chi _{\rho _K} (x) = 0 $ 
if $ x \in K - K_0 $ and that contains the trivial representation as a summand.
\end{lemma}

\begin{proof}
Note that $K/K_0$ is a finite abelian group. Let $\phi$ be the 
regular representation of this group. Since $K/K_0$ is abelian,
we can decompose 
\[ \phi = \bigoplus_{j=1}^m \phi_j \]
where each $\phi_j \colon K/K_0 \to S^1$ is a one-dimensional representation
and $\phi_1$ is the trivial representation. Let $\psi_1 \colon T/K_0
\to S^1$ be the trivial homomorphism. For each $j \geq 2$, since 
$S^1$ is an injective abelian group, $\phi_j$ extends
to a homomorphism $\psi_j \colon T/K_0 \to S^1$. Let $\psi$
be the direct sum of the representations $\psi_j$, which
is an extension of $\phi$ to $T/K_0$. Now we define $\rho_K$ 
to be the composition of $\psi$ and the quotient
map $ \pi \colon T \to T/K_0$. 

The representation $\psi$ contains the trivial representation
$\psi_1$ as a summand, therefore so does $\rho_K$. And if 
$x \in K - K_0 $, then $\pi(x) \neq 1$ and so
\[ \chi_{\rho_K}(x) = \chi_{\psi} (\pi(x)) = \chi_{\phi} (\pi(x)) = 0 \]
as we wanted to prove. 
\end{proof}

Now we use this lemma to show the existence of a faithful
unitary representation of $S$ which is fusion-preserving.

\begin{theorem}
\label{FaithfulFusion}

Let $\Ff$ be a saturated fusion system over a discrete $p$-toral
group $S$. There exists a faithful unitary representation of $S$
which is fusion-preserving.

\end{theorem}

\begin{proof}
By Lemma \ref{EachSubgroup}, for each $ P $ in $\Ff^{cr}$ there is a unitary representation $ \rho _{P \cap T}$ of $T$ 
such that $\chi _{\rho _{P \cap T}} (x) = 0 $ if $ x \in P \cap T - P_0 $. Recall that $ \Aut_{\Ff}(T)$ is a finite 
group and consider
\[  \phi _P = \bigotimes _{w \in \Aut_{\Ff}(T)} w^* \rho _{P \cap T} , \]
where $w^* \rho_{P \cap T} = \rho_{P \cap T} \circ w $. This representation is invariant under the action of $\Aut_{\Ff}(T)$ 
and if $x \in P \cap T - P_0$, then:
\[ \chi _{\phi_P}(x) = \prod _{w \in \Aut_{\Ff}(T)} \chi_{w^*\rho_{P \cap T}}(x) = \chi_{\rho_{P \cap T}}(x) \cdot \prod _{1_T \neq w \in \Aut_{\Ff}(T)} \chi_{w^*\rho_{P \cap T}}(x) = 0 . \]
If $ P $ and $Q$ are conjugate by an element $ s \in S $, the conjugation $c_s$ takes $P \cap T $ to $ Q \cap T $ and $ P_0 $ to $Q_0$. Let $ y \in Q \cap T - Q_0 $ and set $ x = c_s^{-1}(y) \in P \cap T - P_0 $. Then
\[ \chi_{\phi_P}(y) = \chi_{\phi_P}(c_s(x)) = \chi_{\phi_P \circ c_s}(x) = \chi_{\phi_P}(x) = 0 ,\]
since $c_s \in \Aut_{\Ff}(T)$ and $ \phi_P$ is $\Aut_{\Ff}(T)$-invariant. Therefore $\phi_P$ is such that 
$\chi_{\phi _P}(y) = 0 $ if $ y \in Q \cap T - Q_0 $ for any $ Q $ that is $S$-isomorphic to $P$.

Since there are a finite number of $S$-conjugacy classes of subgroups in $\Ff^{cr}$ (Lemma 3.2(a) and Corollary 3.5 in \cite{BLO3}), we can construct
\[ \psi = \bigotimes _{[P] \in \Ff^{cr}} \phi _P \]
where $[P]$ denotes the $S$-conjugacy class of $P$ in $\Ff^{cr}$. It is clear that $ \chi_{\psi}(x) = 0 $ 
if $x \in Q \cap T - Q_0$ for any $ Q $ in $\Ff^{cr}$. By Proposition \ref{InductionTorus}, the existence of a fusion-preserving representation of $S$ follows.

If $ \psi $ is faithful, Proposition \ref{InductionTorus} shows the existence of a faithful fusion-preserving representation. 
Otherwise, given a faithful representation $\alpha$ of $T$, such as the standard representation of $T$, we consider
\[ \alpha_f = \bigoplus _{w \in \Aut_{\Ff}(T)}  w^* \alpha  \]
Note that $ \rho = \psi \otimes \alpha _f $ is invariant under the action of $\Aut_{\Ff}(T)$ and if $y \in P \cap T - P_0 $ 
for some $ P $ in $\Ff^{cr}$ then 
\[ \chi_{\rho}(y) = \chi_{\psi}(y) \chi_{\alpha_f}(y) = 0 \]
Each $\rho_P$ contains the trivial $1$-dimensional representation as a subrepresentation, and therefore so does
$\psi$. The representation $\rho$ is then faithful because it has $\alpha$ as a subrepresentation, which is faithful. 
The result follows from Proposition \ref{InductionTorus}.
\end{proof}  

\begin{proposition}
\label{DirectSummand}
Given a representation $\rho$ of $S$, there exists a fusion-invariant
representation $\alpha$ of $S$ such that $\rho$ is a direct summand
of $\alpha$.
\end{proposition}

\begin{proof}
Let $\psi$ be the representation of $T$ constructed in the proof of 
Theorem \ref{FaithfulFusion}. This representation is invariant under
the action of $\Aut_{\Ff}(T)$ and such that $\chi_{\psi}(x)=0$ if 
$x \in Q \cap T - Q_0 $ for any $ Q $ in $\Ff^{cr}$. Now consider 
the following representation of $T$ 
\[ \beta = \psi \bigotimes \left( \bigoplus _{w \in \Aut_{\Ff}(T)} w^* \res_T^S(\rho) \right) \]
This representation still satisfies the same two conditions as $\psi$
and moreover, it contains $\res _T^S(\rho)$ as a direct summand because
$\psi$ contains the trivial $1$-dimensional representation of $T$ as
a direct summand. By Proposition \ref{InductionTorus}, the representation
$\alpha = \ind_T^S(\beta)$ is fusion-invariant and it contains $\rho$
as a direct summand.
\end{proof}

\section{Embeddings of classifying spaces}
\label{Embeddings}

The aim of this section is to find criteria for the
existence of unitary embeddings of $p$-local compact
groups. 

\begin{definition}
Let $X$ be a topological space. A unitary embedding at $p$ of $X$ is 
a homotopy monomorphism $ X \to BU(N) \pcom $ for some $N > 0 $. If 
$(S,\Ff,\Ll)$ is a $p$-local compact group, a unitary embedding of 
$(S,\Ff,\Ll)$ is a unitary embedding at $p$ of $|\Ll| \pcom$. 
\end{definition}

Since the prime $p$ will be evident from the discussion, from now
on we will simply say unitary embedding.

Because of the homotopy decomposition of the classifying space of
$(S,\Ff,\Ll)$ in terms of the classifying spaces of its
$\Ff$-centric subgroups given in Proposition \ref{Rigidification}, the
restriction map $[|\Ll| \pcom , BU(n) \pcom] \to [BS , BU(n) \pcom]$ 
factors through an inverse limit:
\[ \Psi_n \colon \left[ |\Ll| \pcom , BU(n) \pcom \right] = 
\left[ \hocolim{\Or (\Ff ^c)}{} \widetilde{B}P, BU(n) \pcom \right] \longrightarrow 
\higherlim{\Or (\Ff^c)}{} \left[ \widetilde{B}P,BU(n) \pcom \right] \]
And
\[ \higherlim{\Or (\Ff^c)}{} \left[ \widetilde{B}P,BU(n) \pcom  \right] \cong 
\higherlim{\Or(\Ff ^c)}{} \Rep (P,U(n)) \cong \Rep ^{\Ff}(S,U(n)) \]
where the first isomorphism holds by the argument in the proof of Theorem 1.1 (i) in \cite{JMO2}, which does
not need the discrete $p$-toral group to be an approximation of a $p$-toral group. And the second isomorphism
holds by Remark \ref{AlperinCharacters}. 

We are interested in finding elements in the image of $\Psi _n$ for some $n$ which are faithful representations
because Theorem \ref{HomotopyMonomorphism} clearly implies the following result.

\begin{proposition}
\label{ExistenceEmbeddings}
If there exists a faithful fusion-invariant representation in the
image of $\Psi _n$ for some $n$, then there exists a unitary
embedding of $ (S,\Ff,\Ll) $.
\end{proposition}

The behaviour of the maps $\Psi _n $ was considered in \cite{W} in more generality, where obstructions to the 
surjectivity and injectivity were found. Let $\Ca$ be a small category and $X \colon \Ca \to \Top$ a 
functor. The homotopy colimit of the diagram $X$ is the space
\[ \hocolim{\Ca}{} X = \left( \coprod_{n\geq 0} \coprod_{c_0 \to \cdots \to c_n} X(c_0) \times \Delta^n \right) \Bigg{/} {\sim} , \]
where $\sim$ is induced by the usual face and degeneracy identifications \cite[Ch. XII]{BK}. We introduce a filtration 
of the homotopy colimit induced by the skeleta filtration of the nerve of $\Ca$. Namely, for each $n \geq 0$ we define
\[ F_n X = \left( \coprod_{m \leq n} \coprod_{c_0 \to \cdots \to c_m} X(c_0) \times \Delta^m \right) \Bigg{/} {\sim} , \]
where $\sim$ is induced by the same face and degeneracy identifications as before. Observe that $F_0 X$ is the disjoint 
union of the spaces $X(c)$, where $c$ runs over the set of objects of $\Ca$, and $F_1X$ is the union of the mapping 
cylinders of all maps $X(f)$, where $f$ runs over the set of all morphisms in $\Ca$.

An element of $ \varprojlim _{\Ca} [X(-),Y]$ is given by a set of maps $f_c \colon X(c) \to Y$, one for each object $c$ of $\Ca$ 
together with homotopies $f_{d} \circ X(\varphi) \simeq f_c$ for every morphism $\varphi$ from $c$ to $d$ in $\Ca$. Such
an element defines a map $f_1 \colon F_1 X \to Y$ that satisfies $ f_1|_{X(c)} = f_c$. The strategy to extend $f_1$ up to 
homotopy to a map from the homotopy colimit of $X$ to $Y$ is to extend it by induction on the spaces $F_n X$.

Given a map $\tilde{f}_n \colon F_n X \to Y$ whose restriction to $X(c)$ is homotopic to $f(c)$, Z.~Wojtkowiak developed an obstruction 
theory for extending it to $F_{n+1} X$ without changing it on $F_{n-1} X$. The existence of such an extension depends on the vanishing 
of a certain obstruction class in 
\[ \higherlim{\Ca}{n+1} \pi_n \left( \Map (X(c),Y)_{f(c)} \right) \]
For $n=1$, this expression involves a functor into the category of groups and representations whose $\varprojlim^2$ term is described in 
\cite{W}. When these groups are abelian, the definition of $\varprojlim^2$ coincides with the usual one from homological algebra. Once 
the map has been extended to $F_2 X$, a choice of homotopies allows us to define functors $\pi_n(\Map (X(c),Y)_{f(c)})$ into abelian groups for $n>1$.

Given two maps $\tilde{f}_1,\tilde{f}_2$ from the homotopy colimit of $X$ to $Y$ whose restrictions to $X(c)$ are homotopic to $f(c)$, 
there is also an obstruction theory for the construction of a homotopy $\tilde{f}_1 \simeq \tilde{f}_2$ in \cite{W}.
Clearly, $\tilde{f}_1$ and $\tilde{f}_2$ give rise to a homotopy $\tilde{f}_1|_{F_0 X} \simeq \tilde{f}_2|_{F_0 X}$.
The idea is to extend the homotopy $H_0$ inductively to $I \times F_n X$. Given a homotopy $\tilde{f}_1|_{F_{n-1} X} \simeq \tilde{f}_2|_{F_{n-1} X}$, 
the possibility of extending it to a homotopy between the restrictions of $\tilde{f}_1$ and $\tilde{f}_2$ to $F_n X$ without
changing its values on $F_{n-2}X$ depends  on the vanishing of an obstruction class in $\varprojlim^{n}_{\Ca} \pi_n(\Map (X(c),Y)_{f(c)})$. 

In our case, given $ \rho $ in $ \Rep ^{\Ff} (S,U(n)) $, the obstructions for $\rho$ to be in the image or to have a unique preimage lie 
in higher limits of the functors:
\[ F_i^{\rho} \colon  \Or (\Ff ^c)^{op} \to \Z _{(p)}\text{-Mod} \]
\[  \qquad \qquad \qquad \qquad \qquad \qquad P  \mapsto \pi _i \left( \Map(\widetilde{B}P,BU(n) \pcom)_{\widetilde{B}\rho_{|P}} \right). \]
In this paper we are concerned with the surjectivity of the maps $\Psi_n$ up to stabilization. For this weaker property, 
we found that we can actually refine our obstructions. For that purpose, we start by recalling some definitions from \cite{Lu}.

\begin{definition}
A category $\Or$ is an EI-category if every endomorphism is an isomorphism. An EI-category $\Or$ is finite if the set 
of isomorphism classes of objects is finite and the set of morphisms between any two given objects is finite.
\end{definition} 

Note that $\Or(\Ff ^{cr})$ is a finite EI-category by Lemma 2.5 and Corollary 3.5 in \cite{BLO3}.

\begin{definition}
\label{Length}
Let $\Or$ be a finite EI-category. Given two objects $x_0$ and $x_1$ in $\Or$ we say $x_0 < x_1$
if there is a morphism from $x_0$ to $x_1$ which is not an isomorphism. The length $l(\Or)$ of $\Or$ is the maximum integer $n$
such that there exist $n+1$ different objects $ x_0, x_1, \ldots, x_n $ with $ x_0 < x_1 < \ldots < x_n $.
\end{definition}

\begin{proposition}
\label{BoundedLimits}
Let $\Phi\colon \Or(\Ff^{cr}) \to \Z_{(p)}$-Mod be a functor such that $\Phi(P)$ is a projective $\Z_{(p)}$-module
for any object $P$. Then 
\[ \higherlim{\Or(\Ff^{cr})}i \Phi = 0 \]
for all $ i > l(\Or (\Ff ^{cr})) $.
\end{proposition}

\begin{proof}
The argument of this proof appears in Lemma 5.5 of \cite{N}. Denote the constant functor by  
$ \underline{\Z_{(p)}} \colon \Or (\Ff ^{cr}) \to \Z_{(p)}-\Mod$. By Corollary 5.5 in \cite{BLO3}, there exists $n > 0$ such that
\[  \higherlim{\Or(\Ff^{cr})}j F = 0 \]
for any functor $ F \colon \Or (\Ff ^{cr}) \to \Z_{(p)}-\Mod$ and any $ j > n $. A homological algebra 
argument shows that $\underline{\Z_{(p)}}$ has a finite projective resolution. By Proposition 17.31 in \cite{Lu}, this homological dimension 
is bounded by $l(\Or (\Ff ^{cr}))$.
\end{proof}

Let $ \rho \colon S \to U(n) $ be an $\Ff$-invariant representation. For any $P \leq S$, let $\Irr(P,\rho)$ be the set of 
isomorphism classes of irreducible subrepresentations of $\rho _{|P}$ and let $ R(P,\rho) $ be the subgroup of $R(P)$ generated by $\Irr(P,\rho)$.
Note that for any map $ f\colon P \to Q $ in $\Ff$, the representations $ \rho _{|Q} \circ f$ and $\rho _{|P} $ are isomorphic, and so the induced homomorphism $ R(Q) \to R(P) $ restricts to a group homomorphism $ R(Q,\rho) \to R(P,\rho) $. We will consider the following contravariant functor 
\[ R(-,\rho) \pcom \colon \Or (\Ff ^{cr}) \to \Z _{(p)} -\Mod \]
\[ \qquad \qquad \qquad \qquad \qquad \qquad P  \mapsto R(P,\rho) \otimes _{\Z} \Z \pcom \cong \mathop{\bigoplus} \limits_{\Irr(P,\rho)} \Z \pcom .  \]
Proposition \ref{NaturalIsomorphism} will show that this functor is naturally isomorphic to $F_{2j}^{\rho}$ if the 
representation $\rho$ is faithful and the centralizers of the subrepresentations of $\rho$ 
satisfy a stability condition. 

First we need an auxiliary lemma that extends Proposition 4.4 of \cite{DZ} to discrete $p$-toral 
groups (see also Theorem 1.1 (ii) of \cite{JMO2} and Theorem 5.1 of \cite{N} for analogous results). 
We introduce some notation inspired by this article. For a compact Lie group $G$, let $\B G$ be the 
topological category with one object and $G$ as the space of automorphisms of this object. Given two 
categories $C$ and $D$ and a functor $F\colon C \to D$, let $\Hom(C,D)_F$ be the category whose only object 
is $F$ and whose morphisms are natural transformations from $F$ to $F$. Note that if $G$ and $H$ are
topological groups and $f\colon H \to G $ is a group homomorphism there is an equivalence of topological categories
\[ \B C_G(f(H)) \to \Hom(\B H,\B G)_{\B f} \]
that takes the object of $\B C_G(f(H)) $ to the functor $\B f$, and the morphism $g \in C_G(f(H))$ 
to the natural transformation given by $g \in G$.

Let $P$ be a discrete $p$-toral group and $G$ a compact Lie group such that $\pi_0(G)$ is a finite
$p$-group . Geometric realization of functors induces a map
\[ B\Hom(\B P, \B G)_{\B f} \to \Map (BP,BG)_{Bf} \]
and the $p$-completion map $ \iota \colon BG \to BG \pcom $ induces a map
\[ \Map (BP,BG)_{Bf} \to \Map (BP, BG \pcom) _{Bf \circ \iota} \] 
Let $P$ be the union of an increasing sequence $P_1 \leq P_2 \leq \ldots $ of finite subgroups, as in
Lemma 1.9 of \cite{BLO3}. Since $\pi_0(G)$ is a finite $p$-group, by Proposition 6.22 of \cite{DW}, 
there is an integer $m$ such that if  $ n \geq m $, then the restriction map $ \Map( BP,BG \pcom)_{Bf \circ \iota} \to \Map( BP_n,BG \pcom)_{Bf \circ \iota}$
is a homotopy equivalence. Note that we abusing the notation by using $f$ for the restriction 
to $P_n$. And $\Map( BP_n,BG \pcom)_{B\rho \circ \iota}$ is $p$-complete by \cite[Theorem 3.2]{JMO3},
therefore we obtain a map
\[ \left[ B\Hom( \B P,\B G)_{\B f } \right] \pcom \to \Map ( BP,BG \pcom)_{Bf \circ \iota } \]

\begin{lemma}
\label{ClassifyingMap}
If $P$ is a discrete $p$-toral group, $G$ is a compact Lie group such that $\pi_0(G)$ is a finite $p$-group, and $\rho\colon P \to G$ is a homomorphism, 
then the map constructed above
\[ \left[ B\Hom( \B P,\B G)_{\B \rho } \right] \pcom \to \Map ( BP,BG \pcom)_{B\rho \circ \iota} \]
is a homotopy equivalence.
\end{lemma}

\begin{proof}
Let $P$ be the union of an increasing sequence $P_1 \leq P_2 \leq \ldots $ of finite subgroups, as in
Lemma 1.9 of \cite{BLO3}. Since $G$ is artinian with respect to closed subgroups, there is an integer 
$k$ such that if $ n \geq k $, then $C_G(\rho(P)) = C_G(\rho(P_n)) $. By Proposition 6.22 of \cite{DW},
there exists an integer $m$ such that if $ n \geq m $, then the restriction map 
$ \Map( BP,BG \pcom )_{B\rho \circ \iota} \to \Map( BP_n,BG \pcom )_{B\rho \circ \iota}$ is a homotopy equivalence.

Let $N$ be the maximum of $k$ and $m$ and consider the following commutative diagram
\[
\diagram
\left[ B\Hom(\B P,\B G)_{\B \rho} \right] \pcom \rto \dto& \left[ \Map( BP,BG)_{B\rho} \right] \pcom \rto \dto& \Map( BP,BG \pcom)_{B\rho \circ \iota} \dto\\
\left[ B\Hom(\B P_N,\B G)_{\B \rho} \right] \pcom \rto  &  \left[ \Map( BP_N,BG)_{B\rho} \right] \pcom \rto& \Map( BP_N,BG \pcom )_{B\rho \circ \iota}
\enddiagram
\]
where the vertical maps are induced by the inclusion $ i \colon P_N \to P $ and the horizontal maps by
taking classifying spaces and by the completion maps $ \iota \colon BG \to BG \pcom$. We have already shown that 
the first and last vertical maps are homotopy equivalences. The composite of the lower horizontal 
maps is a homotopy equivalence by \cite[Theorem 3.2]{JMO3}. This proves the lemma.
\end{proof}

\begin{proposition}
\label{NaturalIsomorphism}
Let $\rho \colon S \to U(V) $ be a faithful $\Ff$-invariant representation. Assume the decomposition $\rho = \oplus n_i \mu_i $
in irreducible representations is such that the standard inclusion $U(n_i) \to U(M) $ induces an isomorphism in 
$\pi_{k-1}$ for all $M \geq n_i$ for each $i$. If $k > 0$ is even, then the functors $F_k^{\rho}$ and $R(-,\rho) \pcom $ are naturally isomorphic, and if $k$ is odd, then $F_k^{\rho}$ is the zero functor. 
\end{proposition}

\begin{proof}
Let $\B C_{U(V)}(\rho(P)) \to \Hom(\B P,\B U(V))_{\B \rho}$ be the equivalence of topological categories  
for each $P \leq S $, as we mentioned above. Taking a functor to its geometric realization induces a homotopy equivalence
\[ \left[ B\Hom( \B P , \B U(V) )_{\B \rho} \right] \pcom \to \Map( BP,BU(V)\pcom )_{B\rho}  \]
by Lemma \ref{ClassifyingMap}. The composition $BC_{U(V)}(\rho(P)) \pcom \to \Map( BP,BU(V)\pcom )_{B\rho}$ is
a homotopy equivalence, and we claim that it is natural as functors $\Or (\Ff^{cr}) \to \HoTop$. It is clear that
this composition is a natural transformation of functors $\Ff^{cr} \to \HoTop$, so we only need to check both functors descend
to $\Or (\Ff^{cr})$. Let $a$ be an element of $\Rep_{\Ff}(Q,P)$ represented by $f \colon Q \to P$. This map is determined
by $a$ up to conjugation by an element of $P$ so the map $Bf\colon BQ \to BP$ is determined up to homotopy. By precomposition
we get a map 
\[ \Map( BP,BU(V)\pcom )_{B\rho } \to \Map( BQ,BU(V)\pcom )_{B(f\rho)} = \Map( BQ, BU(V)\pcom )_{B\rho} \]
determined up to homotopy by $a$. On the other hand, $\rho f \rho^{-1} \colon \rho(Q) \to \rho(P)$ is given by conjugation 
by an element of $U(V)$ and determined by $a$ up to conjugation by an element of $\rho(P)$. The induced map $C_{U(V)}(\rho(P))
\to C_{U(V)}(\rho(Q))$ is uniquely determined, since conjugation by an element of $\rho(P)$ is the identity. Therefore
the map $BC_{U(V)}(\rho(P)) \pcom \to BC_{U(V)}(\rho(Q)) \pcom$ is determined by $a$.

By definition, the rigidification $\widetilde{B} \colon \Or (\Ff^{cr}) \to \Top $ is such that there is a natural
homotopy equivalence of functors from $B$ to $\ho \circ \widetilde{B} \colon \Or(\Ff^{cr}) \to \HoTop$. Therefore
there is a homotopy equivalence, natural as functors $\Or (\Ff^{cr}) \to \HoTop$
\[  \Map( \widetilde{B}P,BU(V) \pcom )_{\widetilde{B}\rho}  \to  \Map( BP,BU(V)\pcom )_{B\rho}   \]
In particular, for any $k \geq 1$, there is a natural isomorphism of functors $\Or (\Ff^{cr}) \to \Z_{(p)}-\Mod $
\[ \pi_k \left( BC_{U(V)}(\rho(P)) \pcom \right) \to \pi_k \left( \Map( \widetilde{B}P,BU(V) \pcom )_{\tilde{B}\rho} \right) \]

For each $P \leq S $, we decompose
\[ \rho _{|P} = \bigoplus _{\Irr(P,\rho)} b_i \alpha_i \]
into irreducible representations, and if $n$ is the minimum of the $n_i$, then clearly $b_i \geq n $. 
If $\alpha_i \colon P \to U(W_i)$, then the decomposition can be expressed by Proposition 2.3.15 of \cite{Z1} 
as the isomorphism induced by the evaluation map:
\[ \bigoplus_{W_i \in \Irr(P)} W_i \otimes \Hom_P(W_i,V) \to V , \]
and so $ C_{U(V)}(\rho(P)) \cong \hat{\prod} C_{U(b_i W_i)}(b_i\alpha_i(P)) \cong \hat{\prod} \Aut(\Hom_P(W_i,V)) \cong \hat{\prod} U(b_i)$, where $\hat{\prod}$ denotes the restricted product. Hence we have a natural equivalence	
\[ BC_{U(V)}(\rho(P)) \pcom \simeq \mathop{\prod}^{\wedge}_{\Irr(P,\rho)} BC_{U(b_i W_i)}(b_i \alpha _i(P)) \pcom . \]
We will now understand how maps between centralizers translate into maps of these unitary groups. We have an analogous decomposition 
\[ \rho_{|Q} = \bigoplus _{\Irr(Q,\rho)} d_i \beta_i \]
with $\beta_i \colon Q \to U(Z_i) $ and $d_i \geq n$. Given a map $f\colon Q \to P $, since $\rho$ is $\Ff$-invariant, we have 
$\rho_{|P} \circ f \cong \rho_{|Q}$. Via $f$, the representations $W_i$ decompose as a sum of $Z_k$ and 
so again have an isomorphism
\[ \bigoplus_{W_i \in \Irr(P)} \left( \bigoplus_{Z_k \in \Irr(Q)} Z_k \otimes \Hom_Q(Z_k,W_i) \right) \otimes \Hom_P(W_i,V) \to V \]
Note that the map $f$ induces a commutative diagram
\[
\diagram
C_{U(b_i W_i)}(b_i\alpha_i(P)) \rto & C_{U(V)}(\rho \circ f (Q)) \\
\Aut(\Hom_P(W_i,V)) \rto \uto & \Aut \left( \bigoplus \limits_{Z_k \in \Irr(Q)} \Hom_Q(Z_k,W_i) \otimes \Hom_P(W_i,V) \right) \uto
\enddiagram
\]
where the horizontal maps are induced by $f$ and the vertical maps are inclusions. It is clear that the component of
the lower horizontal that goes from $\Aut(\Hom_P(W_i,V))$ to $\Aut( \Hom_Q(Z_k,W_i) \otimes \Hom_P(W_i,V) )$ is given 
by the diagonal inclusion in blocks.

Let $k=2j$ and fix a generator $ \iota $ of $\pi_{2j-1}(U)$. Let $r$ be the smallest integer such that the standard inclusions
$U(r) \to U(s)$ induce an isomorphism on $\pi_{2j-1}$ for all $s>r$. For each $s \geq r $, we let $\iota_s$ be the
generator of $\pi_{2j-1}(U(s))$ which maps to $\iota $ under the standard inclusion $\pi_{2j-1}(U(s)) \to \pi_{2j-1}(U)$,
which is an isomorphism. Consider the isomorphisms:
\begin{align*}
F_{2j}^{\rho} = \pi_{2j}\left( \Map(\widetilde{B}P,BU(V) \pcom)_{\widetilde{B}\rho_{|P}} \right) & \cong \pi_{2j} \left( BC_{U(V)}(\rho(P) ) \pcom \right) \\
                                      & \cong \pi_{2j} \left( \mathop{\prod}^{\wedge}_{\Irr(P,\rho)} BC_{U(b_i W_i)}(b_i \alpha _i(P)) \pcom \right) \\
                                      & \cong \pi_{2j} \left( \mathop{\prod}^{\wedge}_{\Irr(P,\rho)} BU(b_i) \pcom \right) \\
                                                                & \cong \pi_{2j-1}\left( \mathop{\prod}^{\wedge}_{\Irr(P,\rho)} U(b_i) \pcom \right) \\
                                                                & \cong \bigoplus_{\Irr(P,\rho)} \Z \pcom \\
                                                                & \cong R(P,\rho) \otimes _{\Z} \Z \pcom
\end{align*}
where the fifth isomorphism follows since $b_i \geq n$ and $\pi_{2j-1}(U) \cong \Z $. We claim that this 
composition of isomorphisms is a natural isomorphism. We have already shown that the first two are natural 
isomorphisms, so we only need to show naturality for the composition of the last four isomorphisms. Indeed, 
a morphism $ f \colon Q \to P $ induces the diagonal inclusion of blocks $U(b_i) \to U(m_k b_i)$ for each 
decomposition of $W_i$ into $Z_k$. And this inclusion induces multiplication by $m_k$ on the homotopy groups. 
If we let $\mu_i$ be the image of $\iota_{b_i}$ in $R(P,\rho)$ under the composition of the last
two isomorphisms, then this is mapped to its decomposition into irreducible representations of $Q$ via $f \colon Q \to P $. 
If $k$ is odd, the same chain of isomorphisms gives the zero functor since then $\pi_{k-1}(U)=0$.
\end{proof}

\begin{theorem}
\label{AlmostSurjective}
Let $ \rho \colon S \to U(n) $ be a faithful $\Ff$-invariant representation such that the groups 
$ \! \! \! \! \! \! \higherlim{\Or (\Ff ^{cr})}{2i+1} R(-,\rho) \pcom $ are torsion for 
$ 3 \leq 2i+1 \leq l(\Or (\Ff ^{cr})) $. Then there is a positive integer $m$ such that 
$ m\rho \in \im (\Psi _{mn}) $.
\end{theorem}

\begin{proof}
Let $ \rho _{|S} = n_1 \mu _1 \oplus \ldots \oplus n_s \mu _s $ 
be a decomposition as a direct sum of irreducible representations. Since the homotopy groups of 
unitary groups stabilize, we can choose a positive integer $m_0$ such that the standard inclusions induce isomorphisms 
$  \pi_j(BU(m_0 n_i)) \cong \pi_j(BU(N) )$ for all $ 1 \leq j \leq l(\Or (\Ff ^{cr})) $, for all $ i $ and for 
all $N \geq m_0 n_i $. The representation $m_0\rho$ satisfies the hypothesis of Proposition \ref{NaturalIsomorphism} 
and so $F_k^{m_0\rho}=0$ if $k$ is odd, and $F_k^{m_0\rho}$ is naturally isomorphic to $R(-,\rho) \pcom $ if $k$ is
even.

Assume that we have a map $f_k \colon F_k \tilde{B} \to BU(rn) \pcom $ for some $r \geq 1$ such that 
$f_k|_{\tilde{B}S} \simeq \tilde{B}r\rho$. By the obstruction theory of \cite{W}, the obstruction to 
extend the map to $f_{k+1} \colon F_k \tilde{B} \to BU(rn) \pcom$ lies in 
\[ \higherlim{\Or (\Ff ^{cr})}{k+1} F_k^{r\rho} . \]
If $k$ is odd, these groups are zero. Therefore we may assume that $k$ is even and in this case the 
obstruction group is isomorphic to
\[ \higherlim{\Or (\Ff ^{cr})}{k+1} R(-,\rho) \pcom  . \] 
By assumption, this group is torsion. Let $E$ be its exponent. In the stable range, the map 
$ BU(n) \to BU(n)^s \to BU(sn)$ given by composing the diagonal with the map induced by the 
diagonal inclusion $U(n)^s \to U(sn)$ induces multiplication by $s$ on the homotopy groups. 
Let us denote by $g \colon BU(rn) \pcom \to BU(Ern) \pcom $ the map induced on $p$-completions 
by this process. By postcomposition, we obtain a map $ g \circ f_k \colon F_k \tilde{B} \to BU(Ern) \pcom$ 
such that its restriction to $\tilde{B}S$ is homotopic to $\tilde{B}(Er\rho)$. By the naturality 
of the obstruction theory of \cite{W}, the morphism 
\[ g_* \colon \! \! \! \! \! \! \higherlim{\Or (\Ff ^{cr})}{k+1} R(-,\rho) \pcom \to \! \! \! \! \! \! \higherlim{\Or (\Ff ^{cr})}{k+1} R(-,\rho) \pcom \]
induced by postcomposing with $g$ maps the obstruction for extending $f_k$ to the one for extending 
$ g \circ f_k$. But this morphism is multiplication by $E$, and therefore trivial. Hence there exists 
an extension $f_{k+1} : F_{k+1} \tilde{B} \to BU(Ern) \pcom$ of $g \circ f_k$ up to homotopy which
does not change $g \circ f_k$ on $F_{k-1} \tilde{B}$.

Let $l$ be the length of the category $\Or (\Ff ^{cr})$. Repeating this process a finite number of times 
we obtain a map $f_l \colon F_l \tilde{B} \to BU(mn) \pcom$. Proposition \ref{BoundedLimits} then 
guarantees that $f_l$ extends to a map $f \colon |\Ll| \pcom \to BU(mn) \pcom$ satisfying that $f$ restricted 
to $\tilde{B}S$ is homotopy equivalent to $\tilde{B}(m\rho)$.
\end{proof}

\begin{remark}
\label{Uniqueness}
The question of uniqueness of a preimage of $\tau$ under the maps $\psi_n$
can be treated analogously. Therefore given two maps $f$, $g \colon |\Ll| \pcom \to BU(n) \pcom $ such
that $\psi_n(f) = \psi_n(g) = B\rho$, if the groups 
\[ \higherlim{\Or (\Ff ^{cr})}{2i} R(-,\rho) \pcom \]
are torsion for $ 3 \leq 2i+1 \leq l(\Or (\Ff ^{cr})) $, then there is an integer $ m $ such that 
$mf \simeq mg$. Given a map $h \colon |\Ll| \pcom \to BU(n) \pcom $, we are denoting by $mh$ the 
composition of $h$ with the map $BU(n) \pcom \to BU(mn) \pcom $ induced by the diagonal inclusion 
in blocks. 
\end{remark}

\begin{remark}
\label{SubrepresentationsExtend}
Given a $\Ff$-invariant representation $\rho$, observe that Theorem \ref{AlmostSurjective} only depends 
on the vanishing of obstructions coming from higher limits of the algebraic functor $R(-,\rho) \pcom $.
Note that if the groups in the statement are trivial, the conclusion can be strengthened to say that
there is a positive integer $m$ such that $ k \rho \in \im (\Psi _{kn}) $ if $k \geq m$.

Let $ \mu \colon S \to U(m)$ be an $\Ff$-invariant representation such that $\Irr(S,\mu)$ is a subset of 
$\Irr(S,\rho)$, where $\rho$ is an $\Ff$-invariant representation satisfying that 
\[ \higherlim{\Or (\Ff ^{cr})}{2i+1} R(-,\rho) \pcom = 0 \]
for $ 3 \leq 2i+1 \leq l(\Or (\Ff ^{cr})) $. By the same argument, there exists a positive integer $m$ 
such that $k \rho \in \im (\Psi _{kn}) $ for $k \geq m$, and $ k \rho \oplus \mu$ also extends to a map 
from $|\Ll| \pcom$ for $ k \geq m$. The same observations holds for Remark \ref{Uniqueness}. 
\end{remark}

Now Theorem \ref{AlmostSurjective} and Proposition \ref{ExistenceEmbeddings} imply:

\begin{theorem}
\label{FaithfulMonomorphism}
A $p$-local compact group $(S,\Ff,\Ll)$ with a faithful $\Ff$-invariant representation
that satisfies the hypothesis of Theorem \ref{AlmostSurjective} has a unitary
embedding.
\end{theorem}

Every $p$-local compact group $(S,\Ff,\Ll)$ has a faithful $\Ff$-invariant unitary representation by 
Theorem \ref{FaithfulFusion} and so the following corollary follows:

\begin{corollary} 
\label{LowDepth}
Let $(S,\Ff,\Ll)$ be a $p$-local compact group with $l(\Or (\Ff ^{cr})) < 3 $. Then
it has a unitary embedding.
\end{corollary}

\section{Unitary embeddings of $p$-compact groups}
\label{pcompact}

In this section we apply the results of Section \ref{Embeddings} to show the existence of unitary embeddings of 
$p$-compact groups whose orbit categories of centric radical subgroups have small length, namely
for the Clark-Ewing and the Aguad\'e-Zabrodsky spaces. The existence of unitary embeddings of connected 
$p$-compact groups, also known as the Peter-Weyl theorem for
connected $p$-compact groups (\cite{AGMV}, Theorem 1.6) was proved by using the classification of
$p$-compact groups \cite{AGMV}, \cite{AG}, and showing the existence of such embeddings for the
irreducible exotic $p$-compact groups. This was done for the generalized Grassmannians in \cite{C1}, for
the Aguad\'e-Zabrodsky and Clark-Ewing spaces in \cite{C2} and for $DI(4)$ in \cite{Z1}, \cite{Z15}, \cite{Z2}. The existence for 
the Aguad\'e-Zabrodsky and Clark-Ewing spaces has not been published and we fill this gap in the literature here.   

Recall that a $p$-compact group is a triple $(X,BX,e)$ where $X$ is a space, $BX$ is a $p$-complete
connected pointed space, $H^*(X;\F_p)$ is finite, and $ e \colon X \to \Omega BX $ is a homotopy
equivalence. When there is no danger of confusion we will use $X$ to denote $(X,BX,e)$.

The $p$-completion $ \hat{T} = \Omega (BT \pcom) $  of $ T = (S^1)^r $ is called a $p$-compact torus
of rank $r$. A homomorphism $ f \colon X \to Y $ of $p$-compact groups is a pointed map $ Bf \colon BX \to BY $.
A maximal torus of a $p$-compact group $X$ is a monomorphism $ T \to X $ from a $p$-compact torus of
maximal dimension into $X$. The Weyl group of $X$ is the group of homotopy classes of maps 
$ BT \to BT $ that commute up to homotopy with the inclusion $ BT \to BX $.  

For more information on the foundations of $p$-compact groups from the point of view of homotopy theory, see \cite{DW}. 
We now give some of the details from Section 10 of \cite{BLO3}, where it is shown that every $p$-compact group is 
modeled by a $p$-local compact group.

Given a discrete $p$-toral group $P$, consider the $p$-compact toral group $\hat{P} = \Omega (BP \pcom )$. Let $X$ be a $p$-compact group. A 
discrete $p$-toral subgroup of $X$ is a pair $(P,u)$, where $P$ is a discrete $p$-toral group and $ u\colon\hat{P} \to X $ is a homotopy monomorphism.
Then $X$ has a maximal discrete $p$-toral subgroup $ (S,f)$ and given any discrete $p$-toral subgroup $(P,u)$, it is maximal if and only if
$ p $ does not divide $ \chi ( X / u(\hat{P})) $. The Euler characteristic is taken with respect to homology with coefficients in $\F _p$.

Given such a maximal subgroup $(S,f)$, the fusion system $ \Ff _{S,f}(X) $ is the category with discrete $p$-toral subgroups of $X$ as objects
and morphisms 
\[ \Hom _{\Ff _{S,f}(X)} (P,Q) = \{ \phi \in \Hom (P,Q) \mid Bf _{|BQ} \circ B\phi \simeq Bf _{|BP} \} . \] 
This is a saturated fusion system over $S$ and the centric subgroups correspond to the centric subgroups as defined in 
\cite{CLN}. There is a unique centric linking $ \Ll _{S,f} ^c(X) $ associated to $ \Ff _{S,f}(X) $ such that 
$ | \Ll _{S,f} ^c(X) | \pcom \simeq BX $. If $T$ is the connected component of $S$, the maximal torus of $X$
is $\Omega (BT \pcom)$ and so the Weyl group of $X$ is isomorphic to $ \Aut _{\Ff}(T)$. Since the torus 
of a connected $p$-compact group is self-centralizing, we have $C_S(T) = T $ and in particular $S/T$ is the 
$p$-Sylow subgroup of $ \Aut _{\Ff}(T)$.

We now focus on the families of Clark-Ewing and Aguad\'e-Zabrodsky spaces.

\subsection{Clark-Ewing spaces}

A Clark-Ewing space is a $p$-compact group $X$ such that the order of the Weyl group is prime to $p$. Let 
$\hat{T}$ be a maximal torus of $X$ and $W \leq GL(\pi_1(\hat{T})) $ the Weyl group endowed with a faithful
$p$-adic representation. $W$ acts on $B\hat{T}$ via the above representation. The Clark-Ewing spaces are
homotopy equivalent to the Borel construction $ X = (B\hat{T} \times _W EW) \pcom $.

Since the Weyl group has order prime to $p$, a discrete approximation $ T $ to $\hat{T}$ is a maximal discrete $p$-toral
subgroup of $X$. Since $ T $ is abelian, the only $\Ff$-centric radical subgroup of $T$ is itself
and so the depth of $\Or (\Ff ^{cr}) $ is zero. In this case the fusion system is determined by $\Aut_{\Ff}(T)$, which
equals the Weyl group $W$ of $X$.  By Corollary \ref{LowDepth}, there is a unitary embedding 
of $ BX $. Note that this argument holds for any $p$-local compact group $(S,\Ff,\Ll)$ with
$S$ abelian.

\subsection{Aguad\'e-Zabrodsky spaces}

Let $G_i$ be one of the groups $G_{12}$, $G_{29}$, $G_{31}$, $G_{34}$ from \cite{A}. In this article it is shown that 
these groups have a subgroup isomorphic to $\Sigma _{p_i}$, where $p_{12}=3$, $p_{29} = p_{31} =5$ and $p_{34}=7$. 
Let $\I _i$ be the category with two objects $0$ and $1$ and morphisms $\Hom_{\I _i}(0,0) = G_i$, 
$\Hom_{\I _i}(1,0)= G_i/ \Sigma _{p_i} $, $\Hom_{\I _i}(1,1) = Z(G_i) $ and $\Hom_{\I _i}(0,1) = \emptyset $. 
Consider the functors $ F'_i \colon \I _i \to \HoTop $ defined by $ F'_i(0) = BT^{p_i -1} $, $F'_i(1) = BSU(p_i) $,  
and by $Z(G_i)$ acting via unstable Adams' operations. These functors lift to a functor $ F_i \colon \I _i \to \Top $ 
and the Aguad\'e-Zabrodsky $p$-compact groups $X_{12}$, $X_{29}$, $X_{31}$ and $X_{34}$ are given by:
\[ BX_i = \left( \hocolim{\I _i}{} F_i \right) \pcom . \] 
Note that $X_{12}$ and $X_{31}$ were first described in \cite{Za} using other techniques.

The Weyl group of $ X_i $ is $ G_i $, which has $\Z / p_i $ as its $p_i$-Sylow subgroup. 
Therefore a maximal discrete $p$-toral subgroup of $X_i$ is the semidirect product $S_i$ 
of $ T_i = (\Z /p_i^{\infty})^{p_i-1} $ and $ \Z / p_i $, where $\Z /p_i $ acts on 
$(\Z /p_i^{\infty})^{p_i-1}$ via the inclusion $ \Z /p_i \to \Sigma _{p_i} $ that 
sends $1$ to the $p_i$-cycle $(1 \ldots p_i)$. Let $\sigma$ be a generator of $\Z /p_i$. 
Let us denote by $\Ff_i$ and $\Ll_i$ the fusion system and centric linking systems corresponding to $X_i$. 

The subgroup $S_i$ is $\Ff_i$-centric radical and since $T_i$ is self-centralized, it is $\Ff_i$-centric. 
Now $\Out_{\Ff_i}(T_i)= \Aut_{\Ff_i}(T_i) = G_i $ has a $p_i$-Sylow subgroup of rank $1$ which is not normal 
in $G_i$ and so $T_i$ is $\Ff_i$-centric radical. The center of $S_i$ is the subgroup generated by 
$ (\xi _i, \ldots, \xi_i) \in T $ where $ \xi _i $ is a $p_i$th root of unity.

In order to apply the results of Section \ref{Embeddings}, we need to find the centric radical subgroups of the Aguad\'e-Zabrodsky
spaces. We do so by showing that they coincide with the centric radical subgroups of the centralizer
fusion system of $Z(S_i)$ and identifying this centralizer with the fusion system of $SU(p_i)$, for which 
they are known.

\begin{lemma}
\label{CentricInCentralizer}
Let $\Ff$ be a saturated fusion system over a discrete $p$-toral group $S$, and $Q$
a fully centralized subgroup of $S$. Then a subgroup $ P \leq C_S(Q)$ is $C_{\Ff}(Q)$-centric
if and only if $ P \geq Z(Q)$ and $PQ$ is $\Ff$-centric.
\end{lemma}

\begin{proof}
The proof of the analogous result for fusion systems over finite $p$-groups, that is, 
Proposition 2.5 (a) in \cite{BLO2}, is still valid in this context.
\end{proof}

Recall that if $X$ is a $p$-compact group and $(E,i)$ is a discrete $p$-toral subgroup of $X$,
then $ \Map(BE,BX)_i$ is the classifying space of a $p$-compact group $C_X(E)$ (see Proposition 5.1, 
Theorem 6.1 and Proposition 6.8 in \cite{DW}). If $E$ is a fully centralized subgroup of a $p$-local 
compact group $(S,\Ff,\Ll)$, it is not known in general whether $ |C_{\Ll}(E)| \pcom \simeq 
\Map(BE,|\Ll| \pcom)_{Bi} $, but next proposition shows a particular case when this holds 
at the level of fusion systems.

\begin{proposition}
\label{Centralizers}
Let $X$ be a $p$-compact group, $S$ a maximal discrete $p$-toral subgroup and $\Ff$ the associated fusion
system over $S$. Let $E$ be a fully centralized subgroup of $Z(S)$. Then the fusion system $C_{\Ff}(E)$ 
coincides with the fusion system $\G$ of the $p$-compact group $C_X(E)$ over $C_S(E)$.
\end{proposition}

\begin{proof}
Note that both fusion systems are defined over the same Sylow $C_S(E)$. We first show that they have the
same centric subgroups. Let $P$ be centric in $C_{\Ff}(E)$, then by Lemma \ref{CentricInCentralizer} we have 
that $ P \geq E $ and $P$ is $\Ff$-centric. Since $P$ is $\Ff$-centric, there is a homotopy equivalence 
$\Map (BP,BX)_j \simeq BZ(P)$, where $j \colon BP \to BX $ is the standard inclusion. And then
\begin{align*}
\Map(BP, \Map(BE,BX)_i ) _k & \simeq \Map(BE, \Map(BP,BX)_j ) _l \\
                            & \simeq \Map(BE,BZ(P))_m \\
                            & \simeq BZ(P) 
\end{align*}
where $i$, $k$, $l$ and $m$ are the inclusions. Therefore $P$ is centric in $\G$. Conversely, if $P$
is centric in $\G$, then it contains $Z(C_S(E))$ and so it contains $E$. There is a homotopy equivalence
$ \Map(BP, \Map(BE,BX)_i ) _k \simeq BZ(P)$. Now note that
\[ \Map(BP, \Map(BE,BX)_i ) _k \simeq BC_{C_X(E)}(P) \]
and $ E \leq P $ gives us a map $ BC_X(P) \to BC_X(E) $, which in turn induces $ BC_X(P) \simeq 
BC_{C_X(P)}(P) \to BC_{C_X(E)}(P) $. The inclusion $BC_X(E) \to BX$ gives rise to a map $ BC_{C_X(E)}(P) \to BC_X(P)$,
which is the homotopy inverse. Therefore
\[ BZ(P) \simeq \Map(BP, \Map(BE,BX)_i ) _k \simeq BC_{C_X(E)}(P) \simeq BC_X(P) \]
and so $P$ is centric in $X$.

Now we show that the centric subgroups in $C_{\Ff}(E)$ and $\G$ have the same automorphism groups. 
Let $f$ be a $C_{\Ff}(E)$-automorphism of a centric subgroup $P$. Since $P$ contains $E$, this is an
$\Ff$-automorphism of $P$ which restricts to the identity on $E$. And $E$ is central in $P$, so we
have a diagram which is commutative up to homotopy:
\[
\diagram
BP \times BE \rto \dto_{Bf \times \id} & BX \\
BP \times BE \urto  &
\enddiagram
\]
Taking adjoints, we get another diagram, commutative up to homotopy:
\[
\diagram
BP \rto \dto_{Bf} & \Map (BE,BX)_i \\
BP \urto  &
\enddiagram
\]
which means that $f$ is a $\G$-automorphism of $P$. Conversely, a $\G$-automorphism $g$ of $P$
defines a homotopy commutative diagram as above whose adjoint shows that $g$ is a $C_{\Ff}(E)$-automorphism
of $P$. By Alperin's fusion theorem, morphisms in saturated fusion systems are generated by automorphisms 
of centric subgroups, thus $\G = C_{\Ff}(E)$.
\end{proof} 

The following lemma is a more general version of Lemma 3.8 in \cite{CLN} for $p$-local compact groups.

\begin{lemma}
\label{CentricCentralizer}
Let $\Ff$ be a saturated fusion system over a discrete $p$-toral group $S$, and $P \leq C_S(E)$, where
$E$ be a fully centralized abelian subgroup of $S$. Then $P$ is centric in $ C_{\Ff}(E)$ if 
and only if it is centric in $\Ff$.
\end{lemma}

\begin{proof}
It is clear that centric subgroups in $\Ff$ are centric in $C_{\Ff}(E)$. Let $P$ be centric in $C_{\Ff}(E)$.
By Lemma \ref{CentricInCentralizer}, we have $ P \geq E $ and $ PE = P $ is $\Ff$-centric.
\end{proof}

\begin{lemma}
\label{RadicalCentralizer}
Let $\Ff_i$ be the fusion system over $S_i$ associated to the Aguad\'e-Zabrodsky space $X_i$. Let $Q \neq T_i$ be an $\Ff_i$-centric
subgroup of $S_i$. Then $Q$ is radical in $ C_{\Ff_i}(Z(S_i))$ if and only if it is radical in $\Ff_i$.
\end{lemma}

\begin{proof}
If $Q$ is centric, $Q$ can not be contained in $T_i$. Therefore, $Q$ must be a semidirect product of 
$ Q \cap T_i \leq T_i $ and $ \Z / p_i $. Since $Q$ is centric, we must have $Z(S_i) \leq Q $ and so $ Z(S_i) \leq
Q \cap T_i $. Recall that $ Z(S_i) \cong \Z / p_i$. If we had $ Z(S_i) = Q \cap T_i$, then $Q = \Z / p_i \times \Z /p_i$, 
but in this case $Q$ is $\Ff$-subconjugate to $T_i$ and it can not be centric. Therefore $ Z(S_i) \lneq Q \cap T_i$ 
and so $ Z(Q) = (Q \cap T_i)^{\Z /p_i} = Z(S_i) $. 

Restriction induces a map $ \Aut _{\Ff_i}(Q) \to \Aut(Z(Q)) = \Aut (Z(S_i)) = \Out (Z(S_i)) $, whose kernel 
is $ \Aut _{C_{\Ff_i}(Z(S_i))}(Q) $. Since $\Inn(Q)$ is in the kernel of this map, there is an induced map 
$ \Out _{\Ff_i}(Q) \to \Out (Z(S_i)) $ with kernel $ \Out _{C_{\Ff_i}(Z(S_i))}(Q) $. By Lemma 3.12 in \cite{CLN}, 
if $Q$ is radical in $\Ff_i$, then it is radical in $C_{\Ff_i}(Z(S_i))$. On the other hand, if $Q$ is not 
radical in $\Ff_i$, let $B$ be a nontrivial normal $p$-subgroup of $\Out_{\Ff_i}(Q)$. Its restriction to 
$ \Out (Z(S_i))$ must be trivial because $ \Out (Z(S_i)) \cong \Z / (p_i-1) $, therefore $B$ is a nontrivial 
normal $p$-subgroup of $\Out_{C_{\Ff_i}(Z(S_i))}(Q)$.
\end{proof}

Now we use Lemma \ref{CentricCentralizer} with $ E = Z(S_i)$ and Lemma \ref{RadicalCentralizer}, and so the only 
other centric radical subgroups of $\Ff _i $ must be centric radical in the fusion system $C_{\Ff_i}(Z(S_i))$. 
This fusion system coincides with the fusion system of the $p$-compact group $C_{X_i}(Z(S_i))$ by Proposition 
\ref{Centralizers}.

By Example 4.14.(1) in \cite{M1}, $C_{N(T_i)}(Z(S_i))$ is 
the normalizer of the maximal torus of $C_{X_i}(Z(S_i))$. Since $p_i$ is odd, $N(T_i)$ is 
a semi-direct product and so $C_{N(T_i)}(Z(S_i))$ must be $T_i \rtimes W(SU(p_i)) = N(SU(p_i))$. 
Corollary 7.21 in \cite{M2} says that polynomial $p$-compact 
groups are determined by their maximal torus normalizers and so the $p$-compact group 
$C_{X_i}(Z(S_i))$ is $SU(p_i)$.

From \cite{O}, we know that the 
only other stubborn subgroup of $SU(p_i)$ is the subgroup $ \Gamma_i \leq S_i $ generated by $Z(S_i)$, 
$(1,\xi _i,\xi _i^2, \ldots, \xi _i^{p_i-1})$ and $ \sigma $. Therefore, up to $\Ff_i$-conjugacy, 
the centric radical subgroups of the fusion system $\Ff_i$ associated to the Aguad\'e-Zabrodsky 
$p$-compact group $X_i$ are $T_i$, $S_i$ and $\Gamma_i$. Hence the length of $\Or (\Ff_i ^{cr})$ 
is one. By Corollary \ref{LowDepth}, there is a unitary embedding of $ BX_i $.

We collect the results of this section about unitary embeddings of $p$-compact
groups in the following theorem.

\begin{theorem}
Let $(X,BX,e)$ be a Clark-Ewing space or an Aguad\'e-Zabrodsky space. Then
$BX$ has a unitary embedding.
\end{theorem}

\section{Unitary embeddings of $p$-local compact groups}
\label{Examples}

A recent preprint \cite{BLO4} shows that $p$-completions of 
finite loop spaces are also modeled by $p$-local compact groups. 
More generally, if $ f \colon X \to Y$ is a finite regular covering 
space, where $X$ is the classifying space of a $p$-local compact 
group, then $Y \pcom$ is the classifying space of a $p$-local 
compact group.

If $G$ is a group, we use $ G \wr \Sigma_n$ to denote the semidirect
product of $G^n$ and $\Sigma_n$, where $\Sigma_n$ acts by
permuting the coordinates of $G^n$. Following the notation in \cite{CL}, 
if $X$ is a space we denote by $ X \wr \Sigma_n$ the space 
$X^n \times_{\Sigma_n} E\Sigma_n$ where $\Sigma_n$ acts by permuting the 
coordinates of $X^n$. Recall that $ B(G \wr \Sigma_n) \simeq BG \wr \Sigma_n$.

\begin{proposition}
\label{RegularCovering}
Let $f \colon X \to Y$ be a finite regular covering of a connected
space $Y$, where $X$ is the classifying space of a $p$-local compact 
group. If $X$ has a unitary embedding, then $Y$ has a unitary embedding.
\end{proposition}

\begin{proof}
Let $ X = | \widetilde{\Ll} | \pcom$ for some $p$-local compact group $(\widetilde{S},\widetilde{\Ff},\widetilde{\Ll})$
and let $G$ be the fundamental group of $Y$. By Proposition 7.1 of \cite{BLO4}, there is a fibration 
$ | \widetilde{\Ll} | \to E \to BG $ such that its fiberwise $p$-completion is the fibration 
$ X \to Y \to BG$ associated to the covering $f$. If $\widetilde{\Gamma} = \Aut_{\widetilde{\Ll}}(\widetilde{S})$,
then the classifying map of the fibration $ | \widetilde{\Ll} | \to E \to BG $ defines an extension of groups
\[ 1 \to \widetilde{\Gamma} \to \Gamma \to G \to 1 \]
by the proof of Theorem 7.3 in \cite{BLO4}. By Proposition 5.4 of \cite{BLO4} 
there is a maximal discrete $p$-toral subgroup $S$ of $\Gamma$ which fits into 
an extension of groups $ 1 \to \widetilde{S} \to S \to \pi \to 1 $, where $\pi$ is a finite 
$p$-subgroup of $G$. 

Therefore there is a map of fibrations
\[
\diagram
B\widetilde{S} \rto \dto & BS \rto \dto & B\pi \dto \\
| \widetilde{\Ll} | \rto & E \rto & BG
\enddiagram
\]
where the rightmost vertical map is induced by the inclusion $\pi \to G$. 
By the naturality of the Kahn-Priddy pretransfer constructed in Section 
1 of \cite{KP} there is a commutative diagram
\[ 
\diagram
BS \rto \dto & B\widetilde{S} \wr \Sigma_n \dto \simeq B(\widetilde{S} \wr \Sigma_n) \\
E \rto & | \widetilde{\Ll} | \wr \Sigma_m
\enddiagram
\]
where $n$ and $m$ are the orders of $\pi$ and $G$, respectively. By construction, 
the upper horizontal map is induced by the Kaloujnine-Krasner monomorphism 
$ S \to \widetilde{S} \wr \Sigma_n$ (see \cite{KK} or page 74 of \cite{AM}). 
Let $R_1$ be the image of this monomorphism and note that $ E \pcom \simeq Y \pcom$
by Theorem F.1 of \cite{F}. Then there is a commutative diagram up to homotopy
\[ 
\diagram
BS \rto \dto & BR_1 \dto \\
Y \pcom \rto & ( X \wr \Sigma_m ) \pcom
\enddiagram
\]
where the vertical map on the left is the inclusion of the Sylow $S$
in the $p$-local compact group $Y \pcom$.

Let $ X \to BU(N) \pcom$ be a unitary embedding. By Theorem 6.3(a) of
\cite{BLO3}, it induces a homomorphism $ \rho \colon \widetilde{S} \to S_1$ 
between their Sylows which must be a monomorphism by Theorem 
\ref{HomotopyMonomorphism}. Therefore we have a commutative diagram up to homotopy
\[
\diagram
B(\widetilde{S} \wr \Sigma_n) \rto^{B\alpha} \dto & B(S_1 \wr \Sigma_n) \dto \\
( X \wr \Sigma_m) \pcom \rto & ( BU(N) \pcom \wr \Sigma_m ) \pcom \simeq B(U(N) \wr \Sigma_m) \pcom
\enddiagram
\]
where $ \alpha \colon \widetilde{S} \wr \Sigma_n \to S_1 \wr \Sigma_n$ is induced
by $\rho$ and in particular, it is a monomorphism. Let $ R_2 = \alpha(R_1)$. 
The monomorphism $ U(N) \wr \Sigma_m \to U(mN) $ determines a map 
$ B( U(N) \wr \Sigma_m ) \pcom  \to BU(mN) \pcom$. By Theorem 6.3(a) of \cite{BLO3} 
applied to the composition
\[ BR_2 \to B(S_1 \wr \Sigma_n) \to B(U(N) \wr \Sigma_m) \pcom \to BU(mN) \pcom \]
there is a commutative diagram up to homotopy
\[
\diagram
BR_2 \rto^{B\beta} \dto & BS_2 \dto \\
B(U(N) \wr \Sigma_m) \pcom \rto & BU(mN) \pcom
\enddiagram
\]
for a certain homomorphism $ \beta \colon R_2 \to S_2$, where $S_2$ is a Sylow
of $BU(mN) \pcom$. Since $R_2$ and $S_2$ are subgroups of $U(N) \wr \Sigma_m$ and $U(mN)$, respectively
(see Definition 9.1 of \cite{BLO3}), $\beta$ is a monomorphism.

The composition $ g \colon Y \pcom \to BU(mN) \pcom$ is a map of $p$-local compact groups 
and the induced homomorphism between their Sylows is injective. By Theorem 
\ref{HomotopyMonomorphism}, $g$ is a homotopy monomorphism, that is, a unitary embedding of $Y \pcom$.
\end{proof}

Note that the following theorem depends on the existence of unitary
embeddings for any $p$-compact group.

\begin{theorem}
Let $BX$ be any path connected space such that $ \Omega BX$ is 
$\F_p$-finite. Then $ BX \pcom $ has a unitary embedding.
In particular, if $(X,BX,e)$ is a finite loop space, then $BX \pcom$ has
a unitary embedding.
\end{theorem}

\begin{proof}
Let $X = \Omega BX $. Since $X$ is $\F_p$-finite, $H_0(X;\F_p)$ is finite
and therefore the group $\pi_1(BX) \cong \pi_0(X)$ is finite. So $BX$ is $p$-good 
by Proposition VII.5.1 in \cite{BK}. Consider the universal cover $BZ$ of $BX $ 
which fits into a homotopy fibration sequence $ BZ \to BX \to B\pi_1(X) $ and
the fiberwise $p$-completion $ BZ \pcom \to \widetilde{BX} \to B\pi_1(X)$.

Note that $ Z = \Omega BZ$ is a connected component of $X$ and thus it is $\F_p$-finite.
The space $BZ$ is $p$-good because it is simply connected, therefore $BZ \pcom$ and 
$Z \pcom$ are $p$-complete. Hence $Z$ is $p$-good and $Z \pcom$ is $\F_p$-finite.  
So $BZ \pcom$ is the classifying space of a $p$-compact group. Since $BZ \pcom$ has a unitary embedding, 
so does $ \widetilde{BX} \pcom \simeq BX \pcom$ by Proposition \ref{RegularCovering}.         
\end{proof}

Now we study the two exotic $3$-local compact groups constructed in Section 5.4.3 of \cite{G} 
from the exotic $3$-local finite groups of \cite{DRV}. We will use the results in Section \ref{Embeddings} 
to show the existence of unitary embeddings.

For each $k \geq 1$, consider the finite $3$-group
\[ S_k = \langle s,s_1,s_2 \mid s^3 = s_1^{3^k} = s_2^{3^k} =1, [s_1,s_2] = 1, [s, s_1] = s_2, [s, s_2] = (s_1s_2)^{-3} \rangle \]
and form the union $ S = \cup_{k \geq 1} S_k $ with respect to the inclusions $S_k \to S_{k+1} $ 
that take $s$ to $t$ and $s_i$ to $t_i^3$, where $t$ and $t_i$
are the generators for the analogous presentation of $S_{k+1}$. The group $S$ is a discrete 
$3$-toral group, and it is an extension of $ T = (\Z / 3^{\infty})^2 $ by $ \Z /3 $. Let
$z_1$ and $z_2$ be the elements of order $3$ in $T$ corresponding to $s_1$ and $s_2$
via the inclusion $S_1 \to S $. 

Section 5.4 of \cite{G} shows the existence of two saturated fusion systems $\Ff(2)$ and 
$\Ff(3)$ over $S$, and respective centric linking systems $\Ll(2)$ and $\Ll(3)$.
These fusion systems have three isomorphism classes of centric radical subgroups,
two of which are given by $T$ and $S$. The third is given by the subgroup $E_0 = \langle z_2, s \rangle$ in the
case of $\Ff(2)$ and $V_0 = \langle z_1, z_2, s \rangle $ in the case of $\Ff(3)$. Note that neither $E_0$ nor $V_0$ are
subgroups of $T$, therefore the length of $\Or(\Ff(i)^{cr})$ is one.

\begin{corollary}
There exist unitary embeddings of the $3$-local compact groups $(S,\Ff(2),\Ll(2))$ and $(S,\Ff(3),\Ll(3))$.
\end{corollary}

There is work in progress of A. Gonz\'alez and A. Ruiz to construct generalizations
of these $3$-local compact groups for any other prime $p$, which by construction also
satisfy that the length of the corresponding orbit categories is one.

\section{Homological consequences}
\label{Consequences}

The existence of unitary embeddings of a $p$-local compact group $(S,\Ff,\Ll)$
has consequences on the $p$-local cohomology of $|\Ll| \pcom$ and the Grothendieck ring
of maps from $|\Ll| \pcom $ to $BU(n) \pcom$ for $n \geq 1$, which we discuss in this short section.

\begin{proposition}
\label{Noetherian}
Let $(S,\Ff,\Ll)$ be a $p$-local compact group with a faithful $\Ff$-invariant
representation that satisfies the hypothesis of Theorem \ref{AlmostSurjective}.
Then the rings $ H^*(| \Ll | \pcom ; \Z \pcom ) $ and $ H^*(| \Ll | \pcom ; \F _p) $ 
are Noetherian.
\end{proposition}

\begin{proof}
Let $f\colon |\Ll| \pcom \to BU(n) \pcom $ be a unitary embedding, 
which exists by Theorem \ref{FaithfulMonomorphism}.
Note that the restriction of $f$ to $\widetilde{B}P \pcom$
for each $P \leq S $ is a monomorphism of $p$-compact groups
and so the homotopy fiber $F_P$ of this restriction is $\F_p$-finite. 
By the Proposition in page 180 of \cite{F}, the homotopy fiber $F$ of $f$ is the homotopy colimit
of the spaces $F_P$ over the orbit category $\Or (\Ff)$. This category
has bounded limits at $p$ by Proposition \ref{BoundedLimits}, 
and so the Bousfield-Kan spectral sequence
shows that $F$ is $\F_p$-finite.

$ H^*(|\Ll| \pcom ; \F_p) $ becomes a finitely generated 
$ H^*(BU(n) \pcom ; \F_p)$-module via $f$ by a 
Serre spectral sequence argument. Since $ H^*(BU(n) \pcom ; \F_p)$
is Noetherian, so is $ H^*(|\Ll| \pcom ; \F_p) $. By the Nakayama
lemma, since $H^j(F;\Z \pcom) \otimes _{\Z \pcom} \F_p $ is zero
for $j$ large enough, the same happens for $H^j(F;\Z \pcom)$.
Using again the same Serre spectral sequence argument,
$ H^*(|\Ll| \pcom; \Z \pcom) $ becomes a finitely generated 
$ H^*(BU(n) \pcom ; \Z \pcom)$-module via $f$. 
Since $ H^*(BU(n) \pcom ; \Z \pcom ) $ is Noetherian, so
is $ H^*(X; \Z \pcom ) $.
\end{proof}

\begin{corollary}
Let $(S,\Ff,\Ll)$ be a $p$-local compact group such that at least 
one the following conditions hold:
\begin{enumerate}
\item $l(\Or(\Ff^{cr})) < 3 $.
\item $(S,\Ff,\Ll)$ models a finite loop space or a $p$-compact group. 
\item $(S,\Ff,\Ll)$ is one of the exotic $3$-local compact groups of \cite{G}.
\end{enumerate}
Then $H^*(|\Ll| \pcom; \Z \pcom)$ is Noetherian.
\end{corollary}

For a $p$-compact group, this is the main result in \cite{DW} (combined with \cite{ACFJS}). 
\newline

Let $\K(X)$ be the Grothendieck ring of maps from $X$ to $BU(n) \pcom$ for $n \geq 1$. We make 
the observation that the maps 

\[ \Psi_n \colon [ |\Ll| \pcom, BU(n) \pcom] \to \higherlim{\Or(\Ff ^c)}{} \Rep (P,U(n)) \cong \Rep ^{\Ff}(S,U(n)) \]
from Section \ref{Embeddings} assemble to form a map 
\[ \Psi \colon \K (|\Ll| \pcom) \to \higherlim{\Or(\Ff ^c)}{} R(P) . \]
We denote this limit by $R^{\Ff}(S)$ because it coincides with the Grothendieck ring of $\Ff$-invariant 
representations, as we show now. Given $\chi = \alpha_1 - \alpha _2$ in $R^{\Ff}(S)$,
by Proposition \ref{DirectSummand}, there is a faithful $\Ff$-invariant representation
$\rho$ such that $\rho = \alpha_1 \oplus \beta $. But then
\[ \chi = \alpha_1 - \alpha_2 = (\alpha_1 + \beta) - (\alpha_2 + \beta) = \rho - (\alpha_2 + \beta) \]
Since $\rho$ and $\chi$ are $\Ff$-invariant, so is $\alpha_2 + \beta$, and this proves
our claim.

\begin{theorem}
\label{VectorBundles}
Let $(S,\Ff,\Ll)$ be a $p$-local compact group. If $l(\Or(\Ff^{cr}))<3$, then
$\Psi$ is surjective. If $l(\Or(\Ff^{cr}))<2$, then $\Psi$ is an isomorphism.
\end{theorem}

\begin{proof}
Let us assume $l(\Or(\Ff^{cr}))<3$ and let $\chi \in R^{\Ff}(S)$. Then $\chi
= \alpha_1 - \alpha_2$, where $\alpha_i$ is an $\Ff$-invariant. By Proposition
\ref{DirectSummand}, there exist faithful $\Ff$-invariant representations $\rho_i$ 
such that $\alpha_i$ is a subrepresentation of $\rho_i$. 
Let $\rho = \rho_1 \oplus \rho_2$. By Theorem \ref{AlmostSurjective} and Remark 
\ref{SubrepresentationsExtend}, there is an
integer $M_0$ such that $\alpha_i \oplus M \rho $ and $M \rho$ belong to the image 
of $\Psi$ if $M \geq M_0$. Then
\[ \alpha_1 - \alpha_2 = (\alpha_1 + M_0 \rho) - (\alpha_2 + 2M_0 \rho) + M_0 \rho \]
belongs to the image of $\psi$.

Now let $l(\Or(\Ff^{cr})) < 2 $ and consider $f$, $g \colon |\Ll| \pcom \to BU(n) \pcom $ such
that $ f _{|BS} \simeq g_{|BS} \simeq B\alpha$, where $\alpha$ is some $\Ff$-invariant representation
of $S$. By Proposition \ref{DirectSummand}, there is a faithful $\Ff$-invariant representation $\rho$ such
that $\alpha$ is a subrepresentation of $\rho$. By Remark \ref{SubrepresentationsExtend}, there is an
integer $N$ such $ \alpha \oplus N \rho $ and $N \rho$ belong to the image of $\Psi$ 
and have a unique preimage. Let $ N\rho = \Psi(h) $. Then $\Psi(f+h)=\Psi(g+h) = \alpha \oplus N\rho$, 
from where $ f + h \simeq g + h $ and so $f = g $ in $\K(|\Ll| \pcom)$.
\end{proof}

\begin{remark}
This theorem applies to the $p$-local compact groups
which model the Clark-Ewing or the Aguad\'e-Zabrodsky $p$-compact groups
and to the exotic $3$-local compact groups of \cite{G}. 
\end{remark}


\bibliographystyle{plainnat}


\end{document}